\documentclass[11pt]{article}

\usepackage{amssymb}
\usepackage{amsmath}
\usepackage{theorem}
\usepackage{epsfig}
\usepackage{verbatim}
\usepackage{graphicx}
\usepackage{subfigure}

\usepackage{color}

\newtheorem{theorem}{Theorem}[section]
\newtheorem{lemma}[theorem]{Lemma}

\newtheorem{prop}[theorem]{Proposition}
\newtheorem{corollary}[theorem]{Corollary}

\newtheorem{defi}[theorem]{Definition}
\newtheorem{rem}[theorem]{Remark}

\newcommand{\R}{\Bbb{R}}
\newcommand{\C}{\Bbb{C}}

\newcommand{\T}{\mathbb{T}}
\newcommand{\F}{\mathcal{F}}

\newcommand{\grad}{\nabla}

\newcommand{\dpt}{\partial_t}
\newcommand{\da}{\partial_{\alpha}}

\newcommand{\la}{\Lambda}

\newcommand{\al}{\alpha}
\newcommand{\ep}{\varepsilon}

\newcommand{\ztil}{\tilde{z}}
\newcommand{\dpa}{\partial^{\bot}_{\alpha}}

\newcommand{\pa}{\partial}

\newcommand{\vp}{\varphi}
\newcommand{\om}{\omega}

\newcommand{\be}{\beta}

\newcommand{\co}{\phi_{\delta} * \phi_{\delta} *}

\newenvironment{proof}{\begin{trivlist} \item[] {\em Proof:}}{\hfill $\Box$
                       \end{trivlist}}

\makeatletter
\renewcommand*\l@section{\@dottedtocline{1}{0em}{1.5em}}
\renewcommand*\l@subsection{\@dottedtocline{2}{1.5em}{2.3em}}
\renewcommand*\l@subsubsection{\@dottedtocline{3}{3.8em}{3.7em}}
\makeatother

\numberwithin{equation}{section}

\begin{document}

\title{Finite time singularities for water waves with surface tension}
\date{\today}
\author{ Angel Castro, Diego C\'ordoba, Charles Fefferman, \\ Francisco Gancedo and Javier G\'omez-Serrano \\ \\
\textit{Dedicated to Peter Constantin on his 60th Birthday}}

\maketitle


\begin{abstract}

Here we consider the 2D free boundary incompressible Euler equation with surface tension. We prove that the surface tension does not prevent a finite time splash or splat singularity,  i.e. that the curve touches itself either in a point or along an arc. To do so, the main ingredients of the proof are a transformation to desingularize the curve and a priori energy estimates.


\end{abstract}




\section{Introduction}
In this paper we continue the work in \cite{Castro-Cordoba-Fefferman-Gancedo-GomezSerrano:finite-time-singularities-Euler} and \cite{Castro-Cordoba-Fefferman-Gancedo-GomezSerrano:splash-water-waves} where we show the formation of singularities for the free boundary incompressible Euler equations. Here we prove that in two space dimensions the free boundary problem develops finite time ``splash" and ``splat" singularities when surface tension is taken into account (see below, in Section \ref{SectionInitialData}, the precise definition of the splash and splat curves).

In order to describe the evolution of a fluid with a moving domain $\Omega(t)\subset\R^2$, the 2D incompressible Euler equations are used:
\begin{equation}\label{Euler}
(v_t+v\cdot\grad v)(x,y,t)=-\grad p(x,y,t)-(0,1),\quad (x,y)\in\Omega(t)
\end{equation}
with the fluid velocity $v(x,y,t)\in\R^2$ and the pressure $p(x,y,t)\in\R$. The vector $-(0,1)$ represents the external gravitational force (the acceleration due to gravity is taken equal to one for the sake of simplicity).
The free boundary
\begin{equation}\label{Parametriza}
\partial\Omega(t)=\{z(\alpha,t)=(z_1(\alpha,t),z_2(\alpha,t)):\alpha\in\R\}
\end{equation}
is smooth and convected by the velocity field
\begin{equation}\label{VeloFrontera}
z_t(\alpha,t)\cdot z_\al^\bot(\alpha,t)=v(z(\alpha,t),t)\cdot z_\al^\bot(\alpha,t),
\end{equation}
which is assumed to be incompressible and irrotational
\begin{equation}\label{IncomIrro}
\grad\cdot v(x,y,t)=0,\quad \grad^{\bot}\cdot v(x,y,t)=0, \qquad (x,y)\in\Omega(t).
\end{equation}
Here we study the relevance of considering the Laplace-Young condition for which the pressure on the interface $\partial\Omega(t)$ is proportional to its curvature, meaning that the surface tension effect is considered:
\begin{equation}\label{L-Y}
-p(z(\al,t),t)=\frac{\tau}{2} \frac{z_{\al \al}(\al,t)\cdot z_\al^\bot(\al,t)}{|z_\al(\al,t)|^3} \equiv \frac{\tau}{2}K.
\end{equation}
Above $\tau>0$ is the surface tension coefficient.

The results in this paper can be shown for three different scenarios:
\begin{enumerate}
\item $\Omega(t)$ a compact domain: $z(\al,t)$ is a $2\pi$-periodic function in $\alpha$.
\item Asymptotically flat case: $z(\al,t)-(\al,0)\to 0$ as $\alpha\to\infty$.
\item $\Omega(t)$ periodic in the horizontal variable: $z(\al,t)-(\al,0)$ is a
$2\pi$-periodic function in $\alpha$.
\end{enumerate}

The problem to study here is the potential formation of singularities for the system (\ref{Euler}-\ref{L-Y}) with smooth interface and smooth velocity field with finite energy as initial data:
\begin{align}\label{VeloInicial}
\begin{split}
\Omega(0)&=\Omega_0,\quad \partial\Omega_0=\{z_0(\alpha):\alpha\in\R\},\\
v(x,y,0)&=v_0(x,y),\quad \int_{\Omega_0}|v_0(x,y)|^2dxdy< +\infty.
\end{split}
\end{align}
The smooth initial curve $z_0(\al)$ must satisfy the arc-chord condition:
\begin{equation}\label{arcchord}
 |z_0(\al)-z_0(\beta)|\geq c_{AC}|\al-\beta|,\quad \mbox{for all }\al,\,\beta\in\R,
\end{equation}
where $c_{AC}>0$ is the arc-chord constant.
 The study of this quantity has been employed by other authors to prove local existence (see for example \cite{Wu:well-posedness-water-waves-2d}, \cite{Wu:well-posedness-water-waves-3d}). We will quantify how our curve $z(\al)$ satisfies the arc-chord condition through the following quantity

$$
\F(z)=\frac{|\beta|}{|z(\al)-z(\al-\beta)|},\quad \al,\beta\in [-\pi,\pi].
$$

  Throughout the paper we will only focus on scenario 3 for the sake of simplicity. From now on, we will denote $\Omega_0\cap[-\pi,\pi]\times\R$ by $\Omega_0$ by abuse of notation (a fundamental domain in the period).

We establish the main result in the paper for the system (\ref{Euler}-\ref{L-Y}).
\begin{theorem}
Consider $z_0(\al)-(\al,0)\in H^k(\T)$ for $k\geq 5$. Then there exist a family of initial data satisfying \eqref{VeloInicial} and the arc-chord condition \eqref{arcchord} and a time $T_s>0$ such that the interface $z(\al,t) \in H^{k}(\T)$ from the unique smooth solution of the system (\ref{Euler}-\ref{arcchord}) on the time interval $[0,T_s]$ touches itself at a single point (``splash" singularity) or along an arc (``splat" singularity) at time $t=T_s$.
\end{theorem}

These solutions can be extended to the periodic $3D$ setting considering scenarios invariant under translations in one coordinate direction. In \cite{Coutand-Shkoller:finite-time-splash}, Coutand-Shkoller consider additional $3D$ splash and splat singularities. The case with small initial data was treated by Wu in the two dimensional case \cite{Wu:almost-global-wellposedness-2d} and the three dimensional case was studied by Wu \cite{Wu:global-wellposedness-3d} and Germain et al. \cite{Germain-Masmoudi-Shatah:global-solutions-gravity-water-waves-annals}.

For other long time behaviour results see Alvarez-Lannes \cite{AlvarezSamaniego-Lannes:large-time-existence-water-waves}, Castro et al. \cite{Castro-Cordoba-Fefferman-Gancedo-LopezFernandez:rayleigh-taylor-breakdown} and the references therein.

In order to prove this theorem we proceed as in \cite{Castro-Cordoba-Fefferman-Gancedo-GomezSerrano:finite-time-singularities-Euler} and \cite{Castro-Cordoba-Fefferman-Gancedo-GomezSerrano:splash-water-waves}. Using \eqref{IncomIrro} it is easy to declare that $v$ is harmonic in $\Omega(t)$. This fact allows us to introduce the moment $\omega(\al,t)$ by elementary potential theory as follows:
\begin{equation}
 v(x,y,t)=\frac{PV}{2\pi}\int_{\R}\frac{(x-z_1(\beta,t),y-z_2(\beta,t)))^{\bot}}{|(x,y)-z(\beta,t)|^2}\omega(\beta,t)d\beta,
\end{equation}
 where PV denotes principal value at infinity. This moment is also known in the literature as the vorticity amplitude.
Then the system  (\ref{Euler}-\ref{L-Y}) is equivalent to the following evolution equations which are only written in terms of the free boundary $z(\al,t)$ and the amplitude $\omega(\al,t)$:
\begin{equation}\label{em}
z_t(\al,t)=BR(z,\omega)(\al,t)+c(\al,t)z_{\al}(\al,t),
\end{equation}
\begin{align}
\begin{split}\label{cEuler}
\omega_t(\al,t)&=-2BR_t(z,\omega)(\al,t)\cdot
z_{\al}(\al,t)-\Big(\frac{\omega^2}{4|\da z|^2}\Big)_{\al}(\al,t) +(c\omega)_{\al}(\al,t)\\
&\quad+2c(\al,t) BR_{\al}(z,\omega)(\al,t)\cdot z_{\al}(\al,t)-2
(z_2)_\al(\al,t)+\tau \left(\frac{z_{\al \al}(\al,t)\cdot z_\al^\bot(\al,t)}{|z_\al(\al,t)|^3}\right)_{\al}
\end{split}
\end{align}
(for details see for example \cite[Section 2]{Cordoba-Cordoba-Gancedo:interface-water-waves-2d}).
Above $BR(z,\omega)$ is the Birkhoff-Rott integral defined by
\begin{equation}\label{BR}
BR(z,\omega)=\frac{1}{2\pi}PV\int_{\R}\frac{(z(\al,t)-z(\beta,t))^{\bot}}{|z(\al,t)-z(\beta,t)|^2}\omega(\beta,t)d\beta,
\end{equation}
and $c(\al,t)$ is arbitrary since the boundary is convected by the normal velocity \eqref{VeloFrontera}.

Local existence in Sobolev spaces was first achieved by Wu \cite{Wu:well-posedness-water-waves-2d} assuming  initially the arc-chord condition. For other variations and results see \cite{Craig:existence-theory-water-waves,Nalimov:cauchy-poisson,Beale-Hou-Lowengrub:growth-rates-linearized,Yosihara:gravity-waves,Wu:well-posedness-water-waves-3d,Christodoulou-Lindblad:motion-free-surface,Lindblad:well-posedness-motion,Coutand-Shkoller:well-posedness-free-surface-incompressible,Shatah-Zeng:geometry-priori-estimates,Zhang-Zhang:free-boundary-3d-euler,Lannes:well-posedness-water-waves,Ambrose-Masmoudi:zero-surface-tension-3d-waterwaves,Ambrose:well-posedness-vortex-sheet,Lannes:stability-criterion,Alazard-Burq-Zuily:water-wave-surface-tension,Alazard-Metivier:paralinearization,Cordoba-Cordoba-Gancedo:interface-water-waves-2d}.

The strategy of the proof of the main result is to establish a local existence theorem from the initial data that has a splash or a splat singularity (notice that the equations are time reversible invariant).
Since the curve self-intersects (failure of the arc-chord condition), it is not clear if the amplitude of the vorticity remains smooth and the meaning of equations (\ref{em}-\ref{cEuler}). In order to deal with these obstacles we use a conformal map
 $$
 P(w)=\Big(\tan\Big(\frac{w}{2}\Big)\Big)^{1/2},\quad w\in\C,
 $$
whose intention is to keep apart the self-intersecting points taking the branch of the square root above passing through those crucial points.  Here $P(z)$ will refer to a 2 dimensional vector whose components are the real and imaginary parts of $P(z_1 + iz_2)$. We also make sure that $\Omega(t)\cup\partial\Omega(t) $ do not contain any singular point of the transformation $P$. Then potential theory helps us to get the following analogous evolution equations for the new curve
$$
\ztil(\al,t)=P(z(\al,t))
$$
and the new amplitude $\tilde{\omega}$:
\begin{align}\label{zeq}
\tilde{z}_{t}(\al,t) & = Q^2(\al,t)BR(\tilde{z},\tilde{\omega})(\al,t) + \tilde{c}(\al,t)\tilde{z}_{\al}(\al,t),\end{align}
\begin{align}\label{eqomega}
\tilde{\omega}_{t}(\al,t)  =& -2 BR_t(\tilde{z},\tilde{\omega})(\al,t) \cdot \tilde{z}_{\al}(\al,t) - |BR(\tilde{z},\tilde{\omega})|^{2} (Q^{2})_{\al}(\al,t) - \Big(\frac{Q^2(\al,t)\tilde{\omega}(\al,t)^2}{4|\tilde{z}_{\al}(\al,t)|^{2}}\Big)_{\al}\nonumber \\
& + 2\tilde{c}(\al,t) BR_\al(\tilde{z},\tilde{\omega}) \cdot \tilde{z}_{\al}(\al,t) + \left(\tilde{c}(\al,t)\tilde{\omega}(\al,t)\right)_{\al}
- 2 \left(P^{-1}_2(\tilde{z}(\al,t))\right)_{\al} \nonumber \\
&+\tau\left(\frac{Q^{3}}{|\tilde{z}_{\al}(\al,t)|^{3}}(\tilde{z}_{\al}^{T}HP_{2}^{-1} \tilde{z}_{\al} \nabla P_{1}^{-1} \cdot \tilde{z}_{\al}
- \tilde{z}_{\al}^{T}HP_{1}^{-1} \tilde{z}_{\al} \nabla P_{2}^{-1} \cdot \tilde{z}_{\al})\right)_{\al} \nonumber\\
&+\tau\left(Q\frac{\tilde{z}_{\al\al}(\al,t)\cdot\tilde{z}_{\al}^{\bot}(\al,t)}{|\tilde{z}_{\al}(\al,t)|^3}\right)_{\al} \end{align}
where
$$Q^2(\al,t) = \left|\frac{dP}{dw}(P^{-1}(\tilde{z}(\al,t)))\right|^{2},$$

and $HP_{i}^{-1}$ denotes the Hessian matrix of $P_{i}^{-1}$, which is the $i$-th ($i = \{1,2\}$) component of the transformation $P^{-1}$.

Here, we choose $\tilde{c}(\al,t)$ in such a way that $|\tilde{z}_{\al}(\al,t)| = A(t)$.
 This particular choice of $\tilde{c}$ was first introduced by Hou et al. in \cite{Hou-Lowengrub-Shelley:removing-stiffness} and was later used by Ambrose \cite{Ambrose:well-posedness-vortex-sheet} and Ambrose-Masmoudi \cite{Ambrose-Masmoudi:zero-surface-tension-2d-waterwaves}.  The choice of $\tilde{c}$ implies

\begin{align*}\tilde{c}(\al,t) & = \frac{\al+\pi}{2\pi}\int_{-\pi}^{\pi}(Q^{2}BR(\tilde{z},\tilde{\omega}))_\beta(\beta,t)\cdot\frac{\tilde{z}_{\beta}(\beta,t)}{|\tilde{z}_{\beta}(\beta,t)|^{2}}d\beta \\
& - \int_{-\pi}^{\al}(Q^{2}BR(\tilde{z},\tilde{\omega}))_\beta(\beta,t)\cdot\frac{\tilde{z}_{\beta}(\beta,t)}{|\tilde{z}_{\beta}(\beta,t)|^{2}}d\beta
\end{align*}

It is easy to check that if we take $Q \equiv 1$ in (\ref{zeq}-\ref{eqomega}) we recover (\ref{em}-\ref{cEuler}).

We also define  the function
\begin{align}\label{varphi}
\tilde{\varphi}(\alpha,t)=\frac{Q^2(\alpha,t)\tilde{\omega}(\alpha,t)}{2| \tilde{z}_\alpha(\alpha,t)|}-\tilde{c}(\alpha,t)| \tilde{z}_\alpha(\alpha,t)|
\end{align}
introduced by Beale et al. for the linear case \cite{Beale-Hou-Lowengrub:growth-rates-linearized} and by Ambrose-Masmoudi for the nonlinear one \cite{Ambrose-Masmoudi:zero-surface-tension-2d-waterwaves}. This function will be used to prove local existence in Sobolev spaces.

In the sections below, we show a local existence theorem based on energy estimates. Section \ref{SectionInitialData} is devoted to provide the appropriate initial data for the splash and splat singularities. In Section \ref{SectionEnergyWRT} we choose an energy which does not need a precise sign on the Rayleigh-Taylor function. In Section \ref{SectionEnergyRT} we choose a different energy that involves the sign of the Rayleigh-Taylor function and the estimates are uniform with respect to the surface tension coefficient. These two energies are based on the ones obtained in the non-tilde domain by Ambrose (\cite{Ambrose:well-posedness-vortex-sheet}) and Ambrose-Masmoudi (\cite{Ambrose-Masmoudi:zero-surface-tension-3d-waterwaves}).

The Rayleigh-Taylor function is given by the following formula
\begin{align}
\begin{split}\label{R-T}
 \sigma  \equiv& \left(BR_{t}(\tilde{z},\tilde{\omega}) + \frac{\tilde{\varphi}}{|\tilde{z}_{\al}|}BR_{\al}(\tilde{z},\tilde{\omega})\right) \cdot \tilde{z}_{\al}^{\perp} + \frac{\tilde{\omega}}{2|\tilde{z}_{\al}|^{2}}\left(\tilde{z}_{\al t} + \frac{\tilde{\varphi}}{|\tilde{z}_{\al}|}\tilde{z}_{\al \al}\right) \cdot \tilde{z}_{\al}^{\perp} \\
& + Q\left|BR(\tilde{z},\tilde{\omega}) + \frac{\tilde{\omega}}{2|\tilde{z}_{\al}|^{2}}\tilde{z}_{\al}\right|^{2}(\nabla Q)(\tilde{z}) \cdot \tilde{z}_{\al}^{\perp}
 + (\nabla P_{2}^{-1})(\tilde{z}) \cdot \tilde{z}_{\al}^{\perp}.
\end{split}
\end{align}

All solutions that we will consider throughout the paper will have finite energy, as discussed in \cite{Castro-Cordoba-Fefferman-Gancedo-GomezSerrano:finite-time-singularities-Euler}. The system satisfies the conservation of the mechanical energy. We define it this way: (not to be confused with the subsequent definitions of some other energies, see sections \ref{SectionEnergyWRT} and \ref{SectionEnergyRT}).

\begin{align*}
\mathcal{E}_S(t) & = \frac{1}{2}\int_{\Omega_f(t)}|v(x,y,t)|^2dxdy + \frac{1}{2}\int_{-\pi}^{\pi}(z_2(\alpha,t))^2\da z_{1}(\alpha,t) d\alpha
 + \frac{\tau}{2} \int_{-\pi}^{\pi} |\da z(\al,t)|d\alpha \\
 &  \equiv \mathcal{E}_k(t) + \mathcal{E}_p(t) + \mathcal{E}_\tau(t),
\end{align*}

where $z(\alpha,t) = (z_{1}(\alpha,t), z_{2}(\alpha,t)), u(\alpha,t) = v(z(\alpha,t),t)$, and $\Omega_f(t)=\Omega(t) \cap [-\pi,\pi] \times \mathbb{R}$ is a fundamental domain in the water region in a period, then it follows that the energy is conserved.

\begin{align}
\frac{d\mathcal{E}_k(t)}{dt} & = \int_{\Omega_f(t)}v(x,y,t)(v_t(x,y,t) + v(x,y,t)\cdot \nabla v(x,y,t))dxdy \nonumber \\
& = \int_{\Omega_f(t)}v(x,y,t)(-\nabla p(x,y,t) - (0,1))dxdy \nonumber \\
& = -\int_{\Omega_f(t)}v(x,y,t)(\nabla (p(x,y,t) + y))dxdy \nonumber \\
& = - \int_{\partial (\Omega_f(t))}v(x,y,t)\cdot \overrightarrow{n} y ds
 + \int_{\partial (\Omega_f(t))}v(x,y,t)\cdot \overrightarrow{n} \frac{\tau}{2} K ds\nonumber \\
& = -\int_{-\pi}^{\pi}z_{2}(\alpha,t)u(\alpha,t)\cdot \da z^\bot(\alpha,t)d\alpha
  +\frac{\tau}{2} \int_{-\pi}^{\pi}u(\alpha,t)\cdot \da z^\bot(\alpha,t)  \frac{\da^{2} z(\al,t) \cdot \da z^{\perp}(\al,t)}{|\da z(\al,t)|^{3}} d\alpha
\end{align}
where we have used the incompressibility of the fluid ($\nabla \cdot v = 0$) and Laplace-Young's condition for the pressure on the interface. Next
\begin{align}
\frac{d\mathcal{E}_p(t)}{dt} & = \int_{-\pi}^{\pi}z_2(\alpha,t)\partial_t z_2(\alpha,t)\da z_{1}(\alpha,t)d\alpha + \frac{1}{2}\int_{-\pi}^{\pi}(z_{2}(\alpha,t))^2
\dpt \da  z_{1}(\alpha,t)d\alpha \nonumber \\
& = \int_{-\pi}^{\pi}z_2(\alpha,t)\partial_t z_2(\alpha,t)\da z_{1}(\alpha,t)d\alpha - \int_{-\pi}^{\pi}z_{2}(\alpha,t)\da z_{2}(\alpha,t)\partial_t z_1(\alpha,t)d\alpha \nonumber \\
& = \int_{-\pi}^{\pi}z_2(\alpha,t)u(\alpha,t)\cdot \da z^\bot(\alpha,t)d\alpha.
\end{align}

\begin{align}
\frac{d\mathcal{E}_\tau(t)}{dt} & = \frac{\tau}{2}\int_{-\pi}^{\pi} \frac{\da z(\al,t) \cdot \da \partial_t z(\al,t)}{|\da z(\al,t)|}d\alpha
= -\frac{\tau}{2}\int_{-\pi}^{\pi} \frac{\da^{2} z(\al,t) \cdot \partial_t z(\al,t)}{|\da z(\al,t)|}d\alpha  \nonumber \\
& = -\frac{\tau}{2}\int_{-\pi}^{\pi} \frac{\da^{2} z(\al,t) \cdot u(\al,t)}{|\da z(\al,t)|}d\alpha
= -\frac{\tau}{2}\int_{-\pi}^{\pi} \frac{\da^{2} z(\al,t) \cdot \da z^{\perp}(\al,t)}{|\da z(\al,t)|^{3}}u(\al,t) \cdot \da^{\perp} z(\al,t) d\alpha
\end{align}

Adding all the derivatives we get the desired result.

\section{Properties of the curvature in the tilde domain}
\label{SectionCurvature}

In this section we will rewrite the term corresponding to the curvature $K(z(\al,t))$ in the new tilde variables $\tilde{z}(\al,t)$.

We will proceed step by step. Let us recall that the curvature is defined by
\begin{align*}
K(\al,t) = \frac{z_{\al \al}(\al,t) \cdot z_{\al}^{\bot}(\al,t)}{|z_{\al}(\al,t)|^{3}}
\end{align*}

We begin with the term $|z_{\al}(\al,t)|^{3}$. We have that
\begin{align*}
|\tilde{z}_{\al}(\al,t)|^{2} = \langle \da P(z(\al,t)), \da P(z(\al,t))\rangle
= \langle \nabla P(z(\al,t)) \cdot z_{\al}(\al,t), \nabla P(z(\al,t)) \cdot z_{\al}(\al,t)\rangle
\end{align*}
Since $P$ and $P^{-1}$ are conformal, by the Cauchy-Riemann equations
\begin{align*}
\nabla P(z(\al,t))^{T} \nabla P(z(\al,t)) = Q^{2}(\al,t)\text{Id}_2,
\end{align*}
that implies that
\begin{align*}
|\tilde{z}_{\al}(\al,t)|^{3} = Q^{3}(\al,t)|z_{\al}(\al,t)|^{3}
\end{align*}

We move to the other term
\begin{align*}
\langle z_{\al \al}(\al,t), z_{\al}^{\bot}(\al,t)\rangle
& = \langle \da \left(\nabla P^{-1}(\tilde{z}(\al,t)) \cdot \tilde{z}_{\al}(\al,t)\right), (\nabla P^{-1}(\tilde{z}(\al,t)) \cdot \tilde{z}_{\al}(\al,t))^{\bot}\rangle \\
& = \langle \nabla P^{-1}(\tilde{z}(\al,t)) \cdot \tilde{z}_{\al \al}(\al,t), (\nabla P^{-1}(\tilde{z}(\al,t)) \cdot \tilde{z}_{\al}(\al,t))^{\bot}\rangle \\
& + \langle \da \left(\nabla P^{-1}(\tilde{z}(\al,t))\right) \cdot \tilde{z}_{\al}(\al,t), (\nabla P^{-1}(\tilde{z}(\al,t)) \cdot \tilde{z}_{\al}(\al,t))^{\bot}\rangle \equiv W + X
\end{align*}

Again, by the Cauchy-Riemann equations
\begin{align*}
W = \frac{1}{Q^2(\al,t)}\langle \tilde{z}_{\al \al}(\al,t), \tilde{z}_{\al}(\al,t)^{\bot}\rangle
\end{align*}

Developing the terms in $X$, we get
\begin{align*}
\left(\nabla P^{-1}(\tilde{z}(\al,t))\right) \cdot \tilde{z}_{\al}(\al,t)
= \left(
\begin{array}{c}
\tilde{z}^{T}_{\al}(\al,t) \cdot HP^{-1}_{1}(\tilde{z}(\al,t)) \cdot \tilde{z}_{\al}(\al,t) \\
\tilde{z}^{T}_{\al}(\al,t) \cdot HP^{-1}_{2}(\tilde{z}(\al,t)) \cdot \tilde{z}_{\al}(\al,t) \\
\end{array}
\right),
\end{align*}

where  $HP^{-1}_{i}$ denotes the Hessian of the $i$-th component of $P^{-1}$ ($i = 1,2$). Hence, we can write $X$ as
\begin{align*}
X & =
- \tilde{z}^{T}_{\al}(\al,t) \cdot HP^{-1}_{1}(\tilde{z}(\al,t)) \cdot \tilde{z}_{\al}(\al,t) \nabla P^{-1}_{2}(\tilde{z}(\al,t)) \cdot \tilde{z}(\al,t) \\
& + \tilde{z}^{T}_{\al}(\al,t) \cdot HP^{-1}_{2}(\tilde{z}(\al,t)) \cdot \tilde{z}_{\al}(\al,t) \nabla P^{-1}_{1}(\tilde{z}(\al,t)) \cdot \tilde{z}(\al,t).
\end{align*}

This means that
\begin{align*}
K(\al,t) = Q(\al,t) \frac{\tilde{z}_{\al \al}(\al,t) \cdot \tilde{z}^{\bot}_{\al}(\al,t)}{|\tilde{z}(\al,t)|^{3}}
+ X(\al,t) \frac{Q(\al,t)^{3}}{|\tilde{z}(\al,t)|^{3}} \equiv Q(\al,t) \tilde{K}(\al,t) + M(\al,t)
\end{align*}

We will now try to simplify further by exploiting the Cauchy-Riemann equations. We can calculate the Hessian and the gradient terms as:

\begin{align*}
P^{-1}_{1,x}(\tilde{z}) = \Re\left(\frac{4\tilde{z}}{1+\tilde{z}^{4}}\right) \equiv \Re(a) \\
P^{-1}_{1,y}(\tilde{z}) = \Re\left(\frac{4i\tilde{z}}{1+\tilde{z}^{4}}\right) \equiv -\Im(a) \\
P^{-1}_{2,x}(\tilde{z}) = \Im\left(\frac{4\tilde{z}}{1+\tilde{z}^{4}}\right) \equiv \Im(a) \\
P^{-1}_{2,y}(\tilde{z}) = \Im\left(\frac{4i\tilde{z}}{1+\tilde{z}^{4}}\right) \equiv \Re(a) \\
P^{-1}_{1,x,x}(\tilde{z}) = \Re\left(\frac{4(1-3\tilde{z}^{4})}{(1+\tilde{z}^{4})^{2}}\right) \equiv \Re(b) \\
P^{-1}_{1,x,y}(\tilde{z}) = \Re\left(\frac{4i(1-3\tilde{z}^{4})}{(1+\tilde{z}^{4})^{2}}\right) \equiv -\Im(b) \\
P^{-1}_{2,x,x}(\tilde{z}) = \Im\left(\frac{4(1-3\tilde{z}^{4})}{(1+\tilde{z}^{4})^{2}}\right) \equiv \Im(b) \\
P^{-1}_{2,x,y}(\tilde{z}) = \Im\left(\frac{4i(1-3\tilde{z}^{4})}{(1+\tilde{z}^{4})^{2}}\right) \equiv \Re(b) \\
\end{align*}

Therefore the Hessians are
\begin{equation*}
HP_{1}^{-1} =
\left(
\begin{array}{cc}
\Re(b) & -\Im(b) \\
-\Im(b) & -\Re(b)
\end{array}
\right), \quad
HP_{2}^{-1} =
\left(
\begin{array}{cc}
\Im(b) & \Re(b) \\
\Re(b) & -\Im(b)
\end{array}
\right),
\end{equation*}

Calculating further:

\begin{equation*}
\tilde{z}_{\al}^{T} HP_{2}^{-1} \tilde{z}_{\al} = \Re(b)(2\tilde{z}_{\al}^{1}\tilde{z}_{\al}^{2}) + \Im(b)((\tilde{z}^{1}_{\al})^2 - (\tilde{z}^{2}_{\al})^2)
\end{equation*}
\begin{equation*}
\tilde{z}_{\al}^{T} HP_{1}^{-1} \tilde{z}_{\al} = \Re(b)((\tilde{z}^{1}_{\al})^2 - (\tilde{z}^{2}_{\al})^2) - \Im(b)(2\tilde{z}_{\al}^{1}\tilde{z}_{\al}^{2})
\end{equation*}
\begin{align*}
X_1 & = \Re(a)\Re(b)(2(\tilde{z}_{\al}^{1})^2\tilde{z}_{\al}^{2}) + \Re(a)\Im(b)((\tilde{z}^{1}_{\al})^2\tilde{z}_{\al}^{1} - (\tilde{z}^{2}_{\al})^2\tilde{z}_{\al}^{1}) \\
& + \Im(a)\Re(b)(-2\tilde{z}_{\al}^{1}(\tilde{z}_{\al}^{2})^2) + \Im(b)\Im(b)((\tilde{z}^{1}_{\al})^2\tilde{z}_{\al}^{2} - (\tilde{z}^{2}_{\al})^2\tilde{z}_{\al}^{2}) \\
X_2 & = \Re(b)\Re(b)((\tilde{z}^{1}_{\al})^2\tilde{z}_{\al}^{2} - (\tilde{z}^{2}_{\al})^2\tilde{z}_{\al}^{2}) + \Re(a)\Im(b)(-2\tilde{z}_{\al}^{1}(\tilde{z}_{\al}^{2})^2) \\
 & + \Im(a)\Im(b)(2(\tilde{z}_{\al}^{1})^2\tilde{z}_{\al}^{2}) + \Im(a)\Re(b)((\tilde{z}^{1}_{\al})^2\tilde{z}_{\al}^{1} - (\tilde{z}^{2}_{\al})^2\tilde{z}_{\al}^{1}) \\
\end{align*}
This means
\begin{align*}
X = X_1 - X_2 & = ((\tilde{z}_{\al}^{1})^2+(\tilde{z}_{\al}^{2})^2)(\tilde{z}_{\al}^{2}(\Re(a)\Re(b)+\Im(a)\Im(b)) + \tilde{z}_{\al}^{1}(\Re(a)\Im(b)-\Im(a)\Re(b))) \\
& \equiv ((\tilde{z}_{\al}^{1})^2+(\tilde{z}_{\al}^{2})^2)\langle G(z), \tilde{z}_{\al}\rangle. \\
\end{align*}

We can see that

\begin{align*}
-\frac{Q_{\al}}{Q^{3}} & = \frac{1}{2}\da\left(\frac{1}{Q^{2}}\right) = \da(\Re(a)^{2} + \Im(a)^2) \\
& = \Re(a)\Re(b)\tilde{z}^{1}_{\al} - \Re(a)\Im(b)\tilde{z}^{2}_{\al}
+ \Im(a)\Im(b)\tilde{z}^{1}_{\al} + \Im(a)\Re(b)\tilde{z}^{2}_{\al} \\
& = \langle G(z), \tilde{z}_{\al}^{\perp} \rangle
\end{align*}

by the Cauchy-Riemann equations.

If we take one derivative in space of $X$, we obtain

\begin{align*}
\da X & =  ((\tilde{z}_{\al}^{1})^2+(\tilde{z}_{\al}^{2})^2)\langle \nabla G(\tilde{z}) \cdot \tilde{z}_{\al}, \tilde{z}_{\al}\rangle + ((\tilde{z}_{\al}^{1})^2+(\tilde{z}_{\al}^{2})^2)\langle G(\tilde{z}), \tilde{z}_{\al \al}\rangle \\
& = ((\tilde{z}_{\al}^{1})^2+(\tilde{z}_{\al}^{2})^2)\langle \nabla G(\tilde{z}) \cdot \tilde{z}_{\al}, \tilde{z}_{\al}\rangle + |\tilde{z}_{\al}|^{3} \tilde{K}\langle G(\tilde{z}), \tilde{z}_{\al}^{\perp}\rangle \\
& = ((\tilde{z}_{\al}^{1})^2+(\tilde{z}_{\al}^{2})^2)\langle \nabla G(\tilde{z}) \cdot \tilde{z}_{\al}, \tilde{z}_{\al}\rangle - |\tilde{z}_{\al}|^{3} \tilde{K}\frac{Q_{\al}}{Q^{3}},
\end{align*}

This implies

\begin{align*}
K = Q\tilde{K} - Q^{3}\frac{X}{|\tilde{z}|^{3}}
\Rightarrow K_{\al} = (Q\tilde{K})_{\al} + \frac{Q^{3}}{|\tilde{z}_{\al}|}\langle \nabla G(\tilde{z}) \cdot \tilde{z}_{\al}, \tilde{z}_{\al}\rangle - \tilde{K}Q_{\al} = (Q\tilde{K})_{\al} + M_1 + M_2
\end{align*}

Later, we will see that the $M_1$ is a low order term and can be absorbed by the energy.

\section{Initial data}
\label{SectionInitialData}

For initial data we are interested in considering a self-intersecting curve in one point. More precisely, we will use as initial data \emph{splash curves} which are defined this way:

\begin{defi}
\label{defsplash}
We say that $z(\al) = (z_1(\al),z_2(\al))$ is a \emph{splash curve} if
\begin{enumerate}
\item $z_{1}(\al) - \al, z_2(\al)$ are smooth functions and $2\pi$-periodic.
\item $z(\al)$ satisfies the arc-chord condition at every point except at $\alpha_1$ and $\alpha_2$, with $\alpha_1 < \alpha_2$ where $z(\al_1) = z(\al_2)$ and $|z_{\al}(\al_1)|, |z_{\al}(\al_2)| > 0$. This means $z(\al_1) = z(\al_2)$, but if we remove either a neighborhood of $\al_1$ or a neighborhood of $\al_2$ in parameter space, then the arc-chord condition holds.
\item The curve $z(\alpha)$ separates the complex plane into two regions; a  connected water region and a vacuum region (not necessarily connected). The water region contains each point $x+iy$ for which y is large negative. We choose the parametrization such that the normal vector $n=\frac{(-\pa_\alpha z_2(\alpha), \pa_\alpha z_1(\alpha))}{|\pa_\alpha z(\alpha)|}$ points to the vacuum region. We regard the interface to be part of the water region.
\item We can choose a branch of the function $P$ on the water region such that the curve $\tilde{z}(\al) = (\tilde{z}_1(\al),\tilde{z}_2(\al)) = P(z(\al))$ satisfies:
\begin{enumerate}
\item $\tilde{z}_1(\al)$ and $\tilde{z}_2(\al)$ are smooth and $2\pi$-periodic.
\item $\tilde{z}$ is a closed contour.
\item $\tilde{z}$ satisfies the arc-chord condition.
\end{enumerate}
We will choose the branch of the root that produces that
$$ \lim_{y \to -\infty}P(x+iy) = -e^{-i \pi/4}$$
independently of $x$.
\item $P(w)$ is analytic at $w$ and $\frac{dP}{dw}(w) \neq 0$ if $w$ belongs to the interior of the water region. Furthermore, $(\pm \pi, 0)$ and $(0,0)$ belong to the vacuum region.
\item $\tilde{z}(\al) \neq q^l$ for $l=0,...,4$, where
\begin{equation}\label{points}
q^0=\left(0,0\right),\quad
q^1=\left(\frac{1}{\sqrt{2}},\frac{1}{\sqrt{2}}\right),\quad
q^2=\left(\frac{-1}{\sqrt{2}},\frac{1}{\sqrt{2}}\right),\quad
q^3=\left(\frac{-1}{\sqrt{2}}, \frac{-1}{\sqrt{2}}\right),\quad
q^4=\left(\frac{1}{\sqrt{2}}, \frac{-1}{\sqrt{2}}\right).
\end{equation}
\end{enumerate}
\end{defi}

Moreover, we will define a \emph{splat curve} as a splash curve but replacing condition (2) by the fact that the curve touches itself along an arc, instead of a point.

Let us note that in order to measure when the transformation $P$ is regular, we need to control the distance to the points $q^{l}$. In order to do so, we introduce the function

$$ m(q^{l})(\al, t) \equiv |\tilde{z}(\al,t)-q^{l}|$$

for $l = 0,\ldots, 4$.

We have performed numerical simulations, as explained in \cite{Castro-Cordoba-Fefferman-Gancedo-GomezSerrano:splash-water-waves} with the following initial data on the non-tilde domain:

$$ z^{0}_{1}(\al) = \al + \frac{1}{4}\left(-\frac{3\pi}{2} - 1.9\right)\sin(\al)+\frac{1}{2}\sin(2\al)+\frac{1}{4}\left(\frac{\pi}{2} - 1.9\right)\sin(3\al)$$
$$ z^{0}_{2}(\al) = \frac{1}{10}\cos(\al) - \frac{3}{10}\cos(2\al) + \frac{1}{10}\cos(3\al)$$

Note that $z\left(\frac{\pi}{2}\right) = z\left(-\frac{\pi}{2}\right)$ (splash). Instead of prescribing an initial condition for $\tilde{\omega}$, we prescribed the normal component of the velocity to ensure a more controlled direction of the fluid. From that we got the initial $\tilde{\omega}(\al,0)$ using the following relations. Let $\psi$ be such that $\nabla^{\perp}\psi = v$ and $\Psi(\al)$ its restriction to the interface. The initial normal velocity is then prescribed by setting
$$ u^{0}_{n}(\al)|z_{\al}(\al)| = \Psi_{\al}(\al) = 3 \cdot \cos(\al) - 3.4 \cdot \cos(2\al) + \cos(3\al) + 0.2\cos(4\al).$$

The reader may easily check that the above $z^{0}_{1}$ and $z^{0}_{2}$ yield a splash curve, i.e. the conditions in Definition \ref{defsplash} are satisfied. See Figure \ref{PictureSplash}.

\begin{figure}[h!]\centering
\includegraphics[scale=0.4]{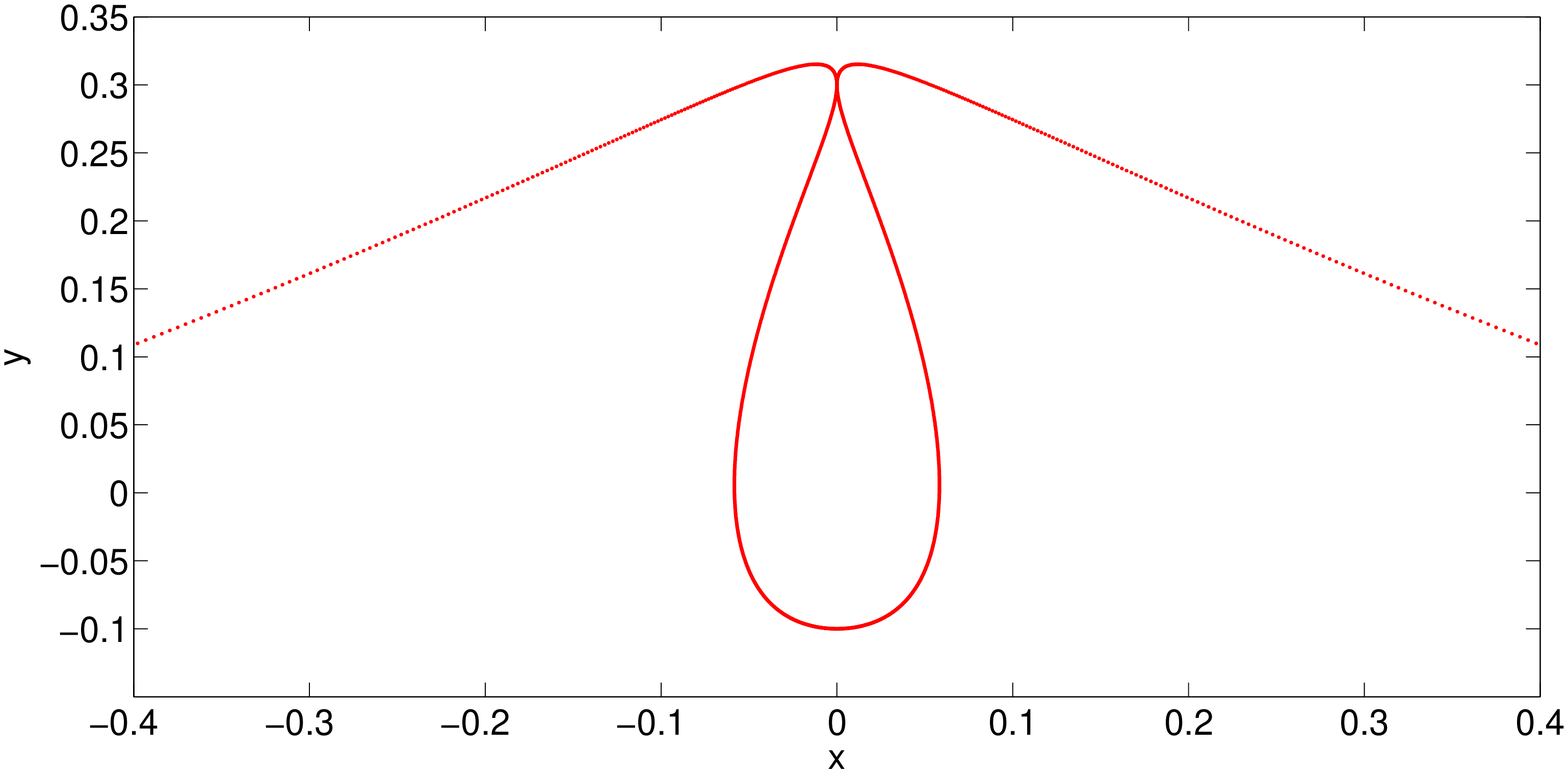}
\caption{Splash singularity. The interface self intersects in a point.}
\label{PictureSplash}
\end{figure}

In order to get an initial data for the splat singularity, one only needs to perturb the splash curve so that it $z_{0}^{1}(\al) = 0$ on a neighbourhood of both $\alpha = \pm \frac{\pi}{2}$. The normal velocity can be the same since it has the right sign (the one that separates the curve). By continuity, the Rayleigh-Taylor function should remain positive.

For the case where the energy is independent on the surface tension coefficient (see Section \ref{SectionEnergyRT}), we need the curve to satisfy the Rayleigh-Taylor condition initially. This is always the case when the surface tension coefficient is small enough. To illustrate this phenomenon, we have plotted in the next figure the Rayleigh-Taylor condition for different values of the surface tension coefficient and the initial condition described above. We can see that for small enough values of $\tau$ (0 and 0.1): the Rayleigh-Taylor condition $\sigma$ is strictly positive. For bigger values of $\tau$, the Rayleigh-Taylor condition $\sigma$ has distinct sign.

\begin{figure}[h!]\centering
\includegraphics[scale=0.4]{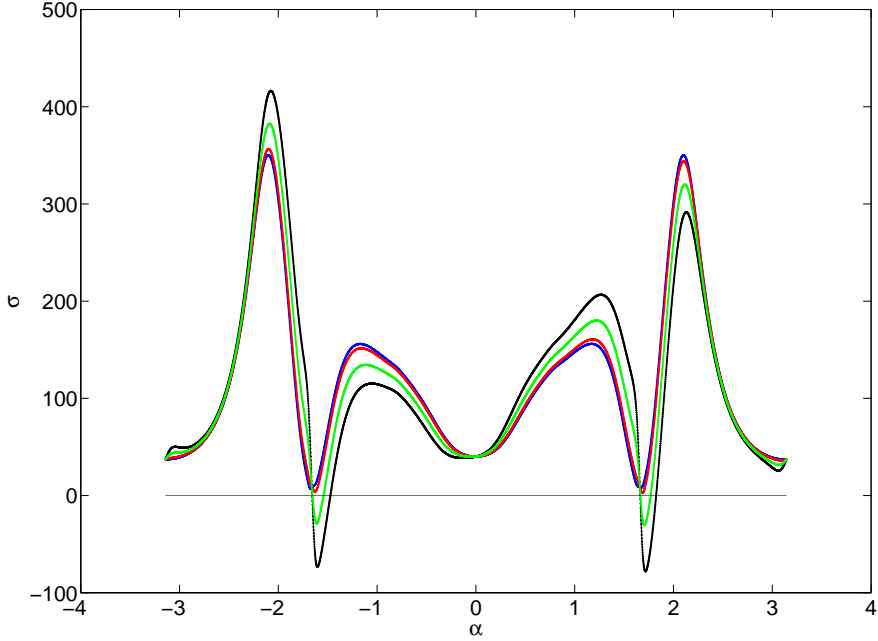}
\caption{Rayleigh-Taylor function for different values of $\tau$: $\tau = 0$ (blue), $\tau = 0.1$ (red), $\tau = 0.5$ (green), $\tau = 1$ (black)}
\label{PictureRTCoeff}
\end{figure}

\section{Energy without the Rayleigh-Taylor condition}
\label{SectionEnergyWRT}
In this section, we prove local existence in the tilde domain, where the time of existence depends on the surface tension coefficient. This theorem has the advantage that the initial data does not need to satisfy the Rayleigh-Taylor condition and it works for every $\tau > 0$.

\begin{theorem}
\label{localexistencetildeWRT}
Let $k \geq 3$. Let $\tilde{z}^{0}(\al)$ be the image of a splash curve by the map $P$ parametrized in such a way that $|\da \tilde{z}^{0}(\al)| = \frac{L}{2\pi}$, where $L$ is the length of the curve in a fundamental period, and such that $\tilde{z}_{1}^{0}(\al), \tilde{z}_2^{0}(\al) \in H^{k+2}(\T)$. Let $\tilde{\omega}(\al,0) \in H^{k+\frac{1}{2}}(\T)$.
Then there exist a finite time $T > 0$, a time-varying curve $\tilde{z}(\al,t)  \in C([0,T];H^{k+2})$, and a function $\tilde{\omega}(\al,t) \in C([0,T];H^{k+\frac{1}{2}})$ providing a solution of the water wave equations (\ref{zeq} - \ref{eqomega}).
\end{theorem}

The proof below is based in the following energy estimates:

\subsection{The energy}\label{primeraenergia}

We will define the energy for $k \geq 3$ as


 $$ E^2_{k}(t) = \mathcal{EE}^{2}(t) + \underbrace{2|\tilde{z}_{\al}|^{3}\int Q^{2k+1}\left(\da^{k}(\tilde{K})\right)^{2}}_{A}
+ \underbrace{\frac{1}{\tau}\int Q^{2k+2}\da^{k}(\tilde{\omega})\Lambda(\da^{k}(\tilde{\omega}))}_{B}
+ \underbrace{\frac{1}{2|\tilde{z}_{\al}|\tau^{2}}\int Q^{2k+3}(\da^{k}(\tilde{\omega}))^{2}\tilde{\omega}^{2}}_{C},$$

 $$\mathcal{EE}^{2}(t) = \|\tilde{z}\|_{L^{2}}^{2} + \|\tilde{\omega}\|_{L^{2}}^{2} +\|\mathcal{F}(z)\|^2_{L^\infty}(t) +\sum_{l=0}^4\frac{1}{m(q^l)(t)},$$

where $m(q^{l})(t) = \min_{\al\in\T}q^{l}(\al,t)$ for $l = 0,\ldots,4$  and $\Lambda = (-\Delta)^{1/2}.$ From now on, we will denote the Hilbert transform of a function $f$ by $H(f)$, where
\begin{align*}
H(f)(\al) = \frac{PV}{\pi}\int_{-\pi}^{\pi}\frac{f(\al-\beta)}{2\tan\left(\frac{\beta}{2}\right)}d\beta.
\end{align*}

Recall that the operator $\Lambda$ can also be written as $\Lambda(f) = \da H(f)$.



\subsection{The energy estimates}

The energy estimates for $\mathcal{EE}$ were proved in \cite{Cordoba-Cordoba-Gancedo:interface-water-waves-2d} and in \cite{Castro-Cordoba-Fefferman-Gancedo-GomezSerrano:finite-time-singularities-Euler}. In this section we will focus on the new terms ($A$, $B$ and $C$).

\subsubsection{$\tilde{K}$}

\begin{prop}
$$ \tilde{K}_t = \text{NICE3 } + \frac{Q^{2}}{2|\tilde{z}_{\al}|^{3}}H(\tilde{\omega}_{\al \al}) + \frac{1}{|\tilde{z}_{\al}|^{3}}(Q^{2})_{\al}H(\tilde{\omega}_{\al}),$$

where NICE3 means $$ \int Q^{j}\da^{k}(\tilde{K}) \da^{k}(NICE3) \leq CE_{k}^{p}(t)$$
for some positive constants $C, p$ and any $j$.
\end{prop}
\begin{proof}
We start writing $\tilde{K}_t$

\begin{align*}
\tilde{K}_t & = \frac{-3}{|\tilde{z}_{\al}|^{5}}\tilde{z}_{\al t} \cdot \tilde{z}_{\al} \tilde{z}_{\al \al} \cdot \tilde{z}_{\al}^{\perp} + \frac{1}{|\tilde{z}_{\al}|^{3}}\left(\tilde{z}_{\al \al t} \cdot \tilde{z}_{\al}^{\perp} + \tilde{z}_{\al \al} \cdot \tilde{z}_{\al t}^{\perp}\right) = P_0 + P_1 + P_2
\end{align*}

Calculating further $P_0$ we get that
\begin{align*}
P_0 & = \frac{-3}{|\tilde{z}_{\al}|^{5}}(Q^{2}BR + c \tilde{z}_{\al})_{\al} \cdot \tilde{z}_{\al} \tilde{z}_{\al \al} \cdot \tilde{z}_{\al}^{\perp} = \text{NICE3},
\end{align*}

by the estimates proved in the Appendix.

On the one hand, developing $P_2$, we obtain
\begin{align*}
P_2 & = \frac{1}{|\tilde{z}_{\al}|^{3}}\left(\tilde{z}_{\al \al} \cdot \tilde{z}_{\al t}^{\perp}\right)
= -\frac{1}{|\tilde{z}_{\al}|^{3}}\left(\tilde{z}_{\al \al}^{\perp} \cdot \tilde{z}_{\al t}\right)=\\
& = -\frac{1}{|\tilde{z}_{\al}|^{3}}\left((Q^{2} BR)_{\al} + (\tilde{c}\tilde{z}_{\al})_{\al}\right) \cdot \tilde{z}_{\al \al}^{\perp} \\
& = \text{NICE3 } -\frac{1}{|\tilde{z}_{\al}|^{3}}\left(\frac{Q^{2}}{2}\frac{\tilde{z}_{\al}^{\perp}}{|\tilde{z}_{\al}|^{2}}H(\tilde{\omega}_{\al}) + \tilde{c}_{\al}\tilde{z}_{\al}\right) \cdot \tilde{z}_{\al \al}^{\perp} = \text{NICE3},
\end{align*}
since $\tilde{c}_{\al}$ is as regular as $\tilde{\omega}, \tilde{z}_{\al \al}$ and therefore bounded in $H^{k}$. On the other, $P_1$ gives rise to
\begin{align*}
P_1 & = \frac{1}{|\tilde{z}_{\al}|^{3}}\left(\tilde{z}_{\al \al t} \cdot \tilde{z}_{\al}^{\perp}\right) = \frac{1}{|\tilde{z}_{\al}|^{3}}\left((Q^{2} BR)_{\al \al} + (\tilde{c}\tilde{z}_{\al})_{\al \al}\right) \cdot \tilde{z}_{\al}^{\perp} = P_{1,1} + P_{1,2}
\end{align*}
We can further develop $P_{1,2}$ to obtain
\begin{align*}
P_{1,2} = \text{NICE3},
\end{align*}
since the terms vanish either by integrating by parts, by being a dot product between two orthogonal vectors or because $\tilde{c}_{\al} = \text{NICE3}$. We also have that
\begin{align*}
P_{1,1} & = \text{NICE3 } + \frac{1}{|\tilde{z}_{\al}|^{3}}\left(2(Q^{2})_{\al} BR_{\al} + Q^{2} BR_{\al \al}\right)\cdot \tilde{z}_{\al}^{\perp} = \text{NICE3 } + P_{1,1,1} + P_{1,1,2}
\end{align*}
The only term in $BR_{\al}$ which is not NICE3 is when we hit with the derivative in $\tilde{\omega}$. Therefore
\begin{align*}
P_{1,1,1} & = \text{NICE3 } + \frac{1}{|\tilde{z}_{\al}|^{3}}2(Q^{2})_{\al} \frac{1}{2}H(\tilde{\omega}_{\al})
\end{align*}
Finally, regarding $P_{1,1,2}$ and keeping in mind that hitting with all the derivatives in $z$ leads us to a term which has the factor $\tilde{z}_{\al \al \al} \cdot \tilde{z}_{\al} = - |\tilde{z}_{\al \al}|^{2}$, giving us the extra regularity we needed to integrate the term.
\begin{align*}
P_{1,1,2} & = \text{NICE3 } + \frac{Q^{2}}{|\tilde{z}_{\al}|^{3}}\tilde{z}_{\al}^{\perp} \cdot \left(\frac{1}{2\pi}\int \frac{(\tilde{z}(\al)-\tilde{z}(\beta))^{\perp}}{|\tilde{z}(\al)-\tilde{z}(\beta)|^{2}}\tilde{\omega}_{\al \al}(\beta)d\beta\right) \\
& = \text{NICE3 } + \frac{Q^{2}}{2|\tilde{z}_{\al}|^{3}}H(\tilde{\om}_{\al \al}).
\end{align*}
We should notice that there doesn't appear a term proportional to $H(\tilde{\om}_{\al})$ since the kernel that results from subtracting the Hilbert transform has room for two derivatives instead of one.

Adding all the previous estimates together we get the desired result.
\end{proof}

\subsubsection{$\tilde{\omega}$}

We first notice that $M_1$ (one of the terms in the curvature) is of the order of $z_{\al}$ and therefore it can be absorbed by the energy. Hence
$$ \da K  = (\tilde{K}Q)_{\al} - \tilde{K}Q_{\al} + \text{low order terms}$$

We will follow the proof done by Ambrose in \cite{Ambrose:well-posedness-vortex-sheet}. Taking into account the estimates for the implicit operator done in \cite{Cordoba-Cordoba-Gancedo:interface-water-waves-2d}, we are left to see the impact of the $Q$ factor in the singular term $(\tilde{c}\tilde{\omega})_{\al}$, since the impact into the others is either trivial (the ones that come from the factor proportional to the curvature) or is zero (the rest of the terms).

\begin{lemma}
$$ \da^{k}(\tilde{c}_{\al} \tilde{\omega}) = \text{NICE35 } + \frac{Q^{2}\tilde{\omega}^{2}}{2|\tilde{z}_{\al}|}H(\da^{k}(\tilde{K})),$$

where NICE35 means  $$ \int Q^{j} \Lambda (\da^{k} (\tilde{\omega}) )\text{NICE35} \leq CE_{k}^{p}(t)$$
for some positive constants $C, p$ and any $j$.
\end{lemma}
\begin{proof}
The most singular term is when we hit all the derivatives in $\tilde{c}_{\al}$, since if we hit all of them in $\tilde{\omega}$, that term would belong to NICE35. Developing the new terms

\begin{align*}
\da^{k}(\tilde{c}_{\al} \tilde{\omega}) & = \text{NICE35 } - \tilde{\omega} \da^{k}\left((Q^{2}BR)_{\al} \cdot \frac{\tilde{z}_{\al}}{|\tilde{z}_{\al}|^{2}}\right) \\
& = \text{NICE35 } - \frac{Q^{2} \tilde{\omega} }{|\tilde{z}_{\al}|^{2}} \cdot  \da^{k}\left(\tilde{z}_{\al} \cdot \int \frac{(\tilde{z}_{\al}(\al) - \tilde{z}_{\al}(\beta))^{\perp}}{|\tilde{z}(\al) - \tilde{z}(\beta)|^{2}}\tilde{\omega}(\beta)d\beta\right) \\
& = \text{NICE35 } -\frac{Q^{2} \tilde{\omega}^{2}}{2|\tilde{z}_{\al}|^{4}}\da^{k}\left(\tilde{z}_{\al} \cdot H(\tilde{z}_{\al \al}^{\perp})\right) \\
& = \text{NICE35 } + \frac{Q^{2} \tilde{\omega}^{2}}{2|\tilde{z}_{\al}|} H\left(\da^{k} (\tilde{K})\right). \\
\end{align*}
\end{proof}
\begin{lemma}
$$ \da^{k}(\tilde{c} \tilde{\omega}_{\al}) = \text{NICE35 }, $$
where NICE35 means  $$ \int Q^{j} \Lambda (\da^{k} (\tilde{\omega}) )\text{NICE35} \leq CE_{k}^{p}(t)$$
for some positive constants $C, p$ and any $j$.
\end{lemma}
\begin{proof}
The most singular term is when we hit all the derivatives in $\tilde{\om}_{\al}$, since if we hit all of them in $\tilde{c}$, that term would belong to NICE35. Thus, we have to estimate

\begin{align*}
\int Q^{j} H \da^{k+1}(\tilde{\omega}) \da^{k+1}(\tilde{\om}) \tilde{c}
& = - \int \da^{k+1}(\tilde{\omega}) H(\da^{k+1}(\tilde{\om}) Q^{j}\tilde{c}) \\
& = \frac{1}{2}\int \da^{k+1}(\tilde{\omega}) \left[H(\da^{k+1}(\tilde{\om})) \tilde{c} Q^{j} - H(\da^{k+1}(\tilde{\om}) \tilde{c} Q^{j})\right]
\leq CE_{k}^{p}(t),
\end{align*}

and therefore it is NICE35.
\end{proof}

\subsection{Calculations of the time derivative of the energy}

Using the previous lemmas and propositions, we can get the following estimates for the derivative of the energy:

\begin{align*}
\frac{dA}{dt} & = \text{OK } +
2\int Q^{2k+1}\left(\da^{k}(\tilde{K})\right)\da^{k}\left(Q^{2}H(\tilde{\omega}_{\al \al}) + 4QQ_{\al}H(\tilde{\omega}_{\al})\right) \\
& = \text{OK } + 2\int 2kQ^{2k+2}Q_{\al}\left(\da^{k}(\tilde{K})\right)\da^{k-1}\left(H(\tilde{\omega}_{\al \al})\right)  \\
& + 2\int Q^{2k+3}\left(\da^{k}(\tilde{K})\right)\da^{k}\left(H(\tilde{\omega}_{\al \al})\right)  \\
& + 2\int 4Q^{2k+2}Q_{\al}\left(\da^{k}(\tilde{K})\right)\da^{k}\left(H(\tilde{\omega}_{\al})\right) \\
& = A^{1} + A^{2} + A^{3},
\end{align*}

where we will say that a term is OK if it is controlled by the energy.

We should be careful while estimating $B_t$ because

\begin{align*}
\frac{dB}{dt} & = \text{OK } + \frac{1}{\tau}\int Q^{2k+2}\da^{k}(\tilde{\omega}_t)\Lambda(\da^{k}(\tilde{\omega}))
+ \frac{1}{\tau}\int Q^{2k+2}\da^{k}(\tilde{\omega})\Lambda(\da^{k}(\tilde{\omega}_t)) \\
& = \text{OK } + \frac{1}{\tau}\int Q^{2k+2}\da^{k}(\tilde{\omega}_t)\Lambda(\da^{k}(\tilde{\omega}))
+ \frac{1}{\tau}\int \Lambda(Q^{2k+2}\da^{k}(\tilde{\omega}))\da^{k}(\tilde{\omega}_t) \\
& = \text{OK } + \frac{2}{\tau}\int Q^{2k+2}\da^{k}(\tilde{\omega}_t)\Lambda(\da^{k}(\tilde{\omega}))
 + \frac{1}{\tau}\int (Q^{2k+2})_{\al}H(\da^{k}(\tilde{\omega}))\da^{k}(\tilde{\omega}_t) \\
\end{align*}

Hence

 \begin{align*}
 \frac{dB}{dt} & = \text{OK } + \frac{2}{\tau}\int Q^{2k+2}\Lambda(\da^{k}(\tilde{\omega}))\frac{Q^{2}\tilde{\omega}^{2}}{2|\tilde{z}_{\al}|}H(\da^{k}(\tilde{K})) \\
& + 2\int Q^{2k+2}\Lambda(\da^{k}(\tilde{\omega}))\da^{k}((Q\tilde{K} + \frac{Q^{3}}{|\tilde{z}_{\al}|^{3}}X)_{\al}) \\
& + \int (2k+2)Q^{2k+1}Q_{\al}H(\da^{k}(\tilde{\omega}))\da^{k+1}((\tilde{K}Q)) \\
& = \text{OK } + \frac{2}{\tau}\int Q^{2k+2}\Lambda(\da^{k}(\tilde{\omega}))\frac{Q^{2}\tilde{\omega}^{2}}{2|\tilde{z}_{\al}|}H(\da^{k}(\tilde{K})) \\
& + 2\int Q^{2k+2}\Lambda(\da^{k}(\tilde{\omega}))\da^{k}((Q\tilde{K})_{\al})
- 2\int Q^{2k+2}Q_{\al}\Lambda(\da^{k}(\tilde{\omega}))\da^{k}(\tilde{K}) \\
& + \int(2k+2) Q^{2k+2}Q_{\al}H(\da^{k}(\tilde{\omega}))\da^{k+1}(\tilde{K}) \\
& = \text{OK } + B^{1} + B^{2} + B^{3} + B^{4}
\end{align*}

$$\frac{dC}{dt}  = \text{OK } + \frac{1}{|\tilde{z}_{\al}|\tau}\int Q^{2k+4}\tilde{\omega}^{2}\da^{k}(\tilde{\omega})\da^{k+1}(\tilde{K}) =
\text{OK } + C^{1}$$

\subsection{Development of the derivative in $B$}

%
%
%
%


We start from the development of $B^1$, $B^2$, $B^3$ and $B^{4}$. We trivially have:

\begin{align*}
B^1 & = \frac{1}{\tau}\int Q^{2k+2}\Lambda(\da^{k}(\tilde{\omega}))\frac{Q^{2}\tilde{\omega}^{2}}{|\tilde{z}_{\al}|}H(\da^{k}(\tilde{K})) \\
B^3 & = - 2\int Q^{2k+2}Q_{\al}\Lambda(\da^{k}(\tilde{\omega}))\da^{k}(\tilde{K}) \\
B^{4} & =  \text{OK } -\int(2k+2) Q^{2k+2}Q_{\al}H(\da^{k+1}(\tilde{\omega}))\da^{k}(\tilde{K})
\end{align*}

We now look at $B^{2}$. We can decompose it in the following way
\begin{align*}
B^{2} & =  2\int Q^{2k+2}\Lambda(\da^{k}(\tilde{\omega}))\da^{k}(Q_{\al}\tilde{K} + Q\tilde{K}_{\al}) \\
& = \text{OK } + 2\int Q^{2k+2}\Lambda(\da^{k}(\tilde{\omega}))(Q_{\al}\da^{k}(\tilde{K}) + Q\da^{k+1}(\tilde{K}) + kQ_{\al}\da^{k}(\tilde{K})) \\
& = \text{OK } + B^{2,1} + B^{2,2} + B^{2,3}
\end{align*}

We can write down the terms $B^{2,1}$ and $B^{2,3}$ in the form
\begin{align*}
B^{2,1} & =  2\int Q^{2k+2}H(\da^{k+1}(\tilde{\omega}))Q_{\al}\da^{k}(\tilde{K}) \\
B^{2,3} & =  2k\int Q^{2k+2}H(\da^{k+1}(\tilde{\omega}))Q_{\al}\da^{k}(\tilde{K})
\end{align*}

Integrating by parts in $B^{2,2}$ we establish

\begin{align*}
B^{2,2} & = -2\int Q^{2k+3}\Lambda(\da^{k+1}(\tilde{\omega}))\da^{k}(\tilde{K}) \\
& -2(2k+3)\int Q^{2k+2}Q_{\al}\Lambda(\da^{k}(\tilde{\omega}))\da^{k}(\tilde{K}) \\
& = B^{2,2,1} + B^{2,2,2}
\end{align*}

Again, $B^{2,2,2}$ can easily be reduced to the canonical form
\begin{align*}
B^{2,2,2} & = -2(2k+3)\int Q^{2k+2}Q_{\al}H(\da^{k+1}(\tilde{\omega}))\da^{k}(\tilde{K})
\end{align*}

%
%
%
%
%

%

%
%

\subsection{Collection of the terms}

We will split all the uncontrolled terms into three categories: high order and low order types I and II and we will see that the sum of the terms in each category adds up to  low enough order terms, denoted by OK. 

\subsubsection{High Order}
\underline{From $A$:}
\begin{equation*}
\begin{array}{lr}
\displaystyle 2\int Q^{2k+3}\left(\da^{k}(\tilde{K})\right)\da^{k}\left(H(\tilde{\omega}_{\al \al})\right) & (A^{2})
\end{array}
\end{equation*}
\underline{From $B$:}
\begin{equation*}
\begin{array}{lr}
\displaystyle -2\int Q^{2k+3}\Lambda(\da^{k+1}(\tilde{\omega}))\da^{k}(\tilde{K}) & (B^{2,2,1})
\end{array}
\end{equation*}
\underline{From $C$:}

No terms from $C$.

\subsubsection{Low Order Type I}
\underline{From $A$:}
\begin{equation*}
\begin{array}{lr}
\displaystyle 2\int 2kQ^{2k+2}Q_{\al}\left(\da^{k}(\tilde{K})\right)\da^{k-1}\left(H(\tilde{\omega}_{\al \al})\right) & (A^{1}) \\
\displaystyle 2\int 4Q^{2k+2}Q_{\al}\left(\da^{k}(\tilde{K})\right)\da^{k}\left(H(\tilde{\omega}_{\al})\right) & (A^{3}) \\
\end{array}
\end{equation*}
\underline{From $B$:}
\begin{equation*}
\begin{array}{lr}
\displaystyle - 2\int Q^{2k+2}Q_{\al}\Lambda(\da^{k}(\tilde{\omega}))\da^{k}(\tilde{K}) & (B^{3}) \\
\displaystyle -\int(2k+2) Q^{2k+2}Q_{\al}H(\da^{k+1}(\tilde{\omega}))\da^{k}(\tilde{K}) & (B^{4}) \\
\displaystyle 2\int Q^{2k+2}H(\da^{k+1}(\tilde{\omega}))Q_{\al}\da^{k}(\tilde{K})) & (B^{2,1})\\
\displaystyle 2k\int Q^{2k+2}H(\da^{k+1}(\tilde{\omega}))Q_{\al}\da^{k}(\tilde{K}) & (B^{2,3}) \\
\displaystyle -2(2k+3)\int Q^{2k+2}Q_{\al}H(\da^{k+1}(\tilde{\omega}))\da^{k}(\tilde{K}) & (B^{2,2,2}) \\
\end{array}
\end{equation*}
\underline{From $C$:}

No terms from $C$.

\subsubsection{Low Order Type II}
\underline{From $A$:}

No terms from $A$.

\underline{From $B$:}

\begin{equation*}
\begin{array}{lr}
\displaystyle \frac{1}{\tau}\int Q^{2k+2}\Lambda(\da^{k}(\tilde{\omega}))\frac{Q^{2}\tilde{\omega}^{2}}{|\tilde{z}_{\al}|}H(\da^{k}(\tilde{K}))
& (B^{1})
\end{array}
\end{equation*}

\underline{From $C$:}

\begin{equation*}
\begin{array}{lr}
\displaystyle \frac{1}{|\tilde{z}_{\al}|\tau}\int Q^{2k+4}\tilde{\omega}^{2}\da^{k}(\tilde{\omega})\da^{k+1}(\tilde{K}) & (C^{1})
\end{array}
\end{equation*}

%
%
%
%
%
%
%
%

\subsection{Regularized system}

Now, let $\tilde{z}^{\ep,\delta,\mu} (\al,t)$ be a solution of the following system (compare with (\ref{zeq} - \ref{eqomega})):
\begin{align}
\label{epsdelmuzt}
\begin{split}\tilde{z}^{\ep,\delta,\mu}_t(\al,t)&=\phi_{\delta} * \phi_{\delta} * \left(Q^2(\tilde{z}^{\ep,\delta,\mu})BR(\tilde{z}^{\ep,\delta,\mu},\tilde{\om}^{\ep,\delta,\mu})\right)(\al,t)+\phi_{\mu} * \left(\tilde{c}^{\ep,\delta,\mu}\left(\phi_{\mu} * \da \tilde{z}^{\ep,\delta,\mu}\right)\right)(\al,t),
\end{split}
\end{align}

\begin{align}
\label{epsdelmuwt}
\tilde{\omega}&^{\ep,\delta,\mu}_{t}=\phi_{\delta} * \phi_{\delta} *\left(-2 BR_t(\tilde{z}^{\ep,\delta,\mu},\tilde{\omega}^{\ep,\delta,\mu}) \cdot \tilde{z}^{\ep,\delta,\mu}_{\al}- |BR(\tilde{z}^{\ep,\delta,\mu},\tilde{\omega}^{\ep,\delta,\mu})|^{2} (Q^{2}(\tilde{z}^{\ep\delta,\mu}))_{\al}\right. \nonumber  \\
 &- \Big(\frac{Q^2\tilde{(\omega^{\ep,\delta,\mu})^2}}{4|\tilde{z}^{\ep,\delta,\mu}_{\al}|^{2}}\Big)_{\al}
+ 2\overline{c}^{\ep,\delta,\mu} BR_\al(\tilde{z}^{\ep,\delta,\mu},\omega^{\ep,\delta,\mu}) \cdot \tilde{z}^{\ep,\delta,\mu}_{\al} + \left(\overline{c}^{\ep,\delta,\mu}\tilde{\omega}^{\ep,\delta,\mu}\right)_{\al}- 2 \left(P^{-1}_2(\tilde{z}^{\ep,\delta,\mu}(\al,t))\right)_{\al} \nonumber \\
&+\tau \left( \frac{Q^{3}(\tilde{z}^{\ep,\delta,\mu})}{|\tilde{z}^{\ep,\delta,\mu}_{\al}(\al,t)|^{3}}
(\tilde{z}^{\ep,\delta,\mu}_{\al})^{T}HP_{2}^{-1} \tilde{z}^{\ep,\delta,\mu}_{\al}\nabla P_{1}^{-1}\cdot \tilde{z}^{\ep,\delta,\mu}_{\al}
-(\tilde{z}^{\ep,\delta,\mu}_{\al})^{T}HP_{1}^{-1} \tilde{z}^{\ep,\delta,\mu}_{\al}\nabla P_{2}^{-1} \cdot \tilde{z}^{\ep,\delta,\mu}_{\al})\right)_{\al}
\nonumber\\
&\left.+\tau\left(Q\frac{\tilde{z}^{\ep,\delta,\mu}_{\al\al}\cdot (\tilde{z}^{\ep,\delta,\mu}_{\al})^{\bot}}{|\tilde{z}^{\ep,\delta,\mu}_{\al}|^3}\right)_{\al}\right)-  \ep\phi_\mu*\phi_\mu*\left(\Lambda (\tilde{\omega}^{\ep,\delta,\mu})\frac{1}{Q^{2k+3}}\right)
\end{align}
$\tilde{z}^{\ep,\delta,\mu}(\al,0)=\tilde{z}_0(\al)$ and $\tilde{\om}^{\ep,\delta,\mu}(\al,0)=\tilde{\om}_0(\al)$ for $\ep>0$, $ \delta > 0, \mu > 0$. The functions $\phi_{\delta}$ and $\phi_{\mu}$ are even mollifiers,

\begin{align*}
\tilde{c}^{\ep,\delta,\mu}(\al)=&\frac{\al+\pi}{2\pi}\int_{-\pi}^\pi\frac{\partial_{\beta} \tilde{z}^{\ep,\delta,\mu}(\be))}{|\partial_{\beta}
\tilde{z}^{\ep,\delta,\mu}(\be)|^2}\cdot \phi_{\delta} * \phi_{\delta} * (\partial_{\beta} (Q^2(\tilde{z}^{\ep,\delta,\mu})(\be) BR(\tilde{z}^{\ep,\delta,\mu},\tilde{\om}^{\ep,\delta,\mu}))(\be)) d\be \\
& -\int_{-\pi}^\al
\frac{\partial_{\beta} \tilde{z}^{\ep,\delta,\mu}(\beta)}{|\partial_{\beta} \tilde{z}^{\ep,\delta,\mu}(\beta)|^2}\cdot\phi_{\delta} * \phi_{\delta} * (\partial_{\beta} (Q^2(\tilde{z}^{\ep,\delta,\mu})(\beta)
BR(\tilde{z}^{\ep,\delta,\mu},\tilde{\om}^{\ep,\delta,\mu}))(\beta)) d\beta,
\end{align*}
and
\begin{align*}
\overline{c}^{\ep,\delta,\mu}(\al)=&\frac{\al+\pi}{2\pi}\int_{-\pi}^\pi\frac{\partial_{\beta} \tilde{z}^{\ep,\delta,\mu}(\be))}{|\partial_{\beta}
\tilde{z}^{\ep,\delta,\mu}(\be)|^2}\cdot  (\partial_{\beta} (Q^2(\tilde{z}^{\ep,\delta,\mu})(\be) BR(\tilde{z}^{\ep,\delta,\mu},\tilde{\om}^{\ep,\delta,\mu}))(\be)) d\be \\
& -\int_{-\pi}^\al
\frac{\partial_{\beta} \tilde{z}^{\ep,\delta,\mu}(\beta)}{|\partial_{\beta} \tilde{z}^{\ep,\delta,\mu}(\beta)|^2}\cdot (\partial_{\beta} (Q^2(\tilde{z}^{\ep,\delta,\mu})(\beta)
BR(\tilde{z}^{\ep,\delta,\mu},\tilde{\om}^{\ep,\delta,\mu}))(\beta)) d\beta,
\end{align*}

The RHS of the evolution equations for $\tilde{z}^{\ep,\delta,\mu}$ and $\tilde{\omega}^{\ep,\delta,\mu}$ are Lipschitz in the spaces $H^{k+2}(\T)$ and $H^{k+\frac{1}{2}}(\T)$ since they are mollified.  Therefore we can solve (\ref{epsdelmuzt}-\ref{epsdelmuwt}) for short time, thanks to Picard's theorem.

Now, we can perform energy estimates  to get uniform bounds in $\mu$ (we just deal with a transport term and a dissipative) and we can let $\mu$ go to zero. The energy estimates that we can get are the following:

\begin{align*}
& \frac{d}{dt}\left(\|\tilde{z}^{\ep,\delta,\mu}\|^2_{H^5}
+\|\F(\tilde{z}^{\ep,\delta,\mu})\|^2_{L^\infty}+\|\tilde{\om}^{\ep,\delta,\mu}\|^2_{H^{3+\frac12}}+\sum_{l=0}^{4}\frac{1}{m^{\ep,\delta,\mu}(q^l)}\right)(t)\\
&\leq
C(\delta)\left(\|\tilde{z}^{\ep,\delta,\mu}\|^2_{H^5}+\|\F(\tilde{z}^{\ep,\delta,\mu})\|^2_{L^\infty}+
\|\tilde{\om}^{\ep,\delta,\mu}\|^2_{H^{3+\frac12}}+\sum_{l=0}^4\frac{1}{m^{\ep,\delta,\mu}(q^l)}\right)^j(t).
\end{align*}
We should note that for the new system without the $\phi_\mu$ mollifier, the length of the tangent vector $|\da \tilde{z}^{\delta}|$ is now constant in space and depends only on time. Next we will perform energy estimates as in the previous case by using the curvature $\tilde{K}^{\delta}$ from the curve $\tilde{z}^{\delta}$.

Similarly, we get (let us omit the superscript $\delta,\varepsilon$ in $\tilde{z}^{\delta,\varepsilon}$ and $\tilde{\om}^{\delta,\varepsilon}$)
\begin{itemize}
\item $$ \tilde{K}_t = \text{NICE3 } + \frac{Q^{2}}{2|\tilde{z}_{\al}|^{3}}\co H(\tilde{\omega}_{\al \al} )+ \frac{1}{|\tilde{z}_{\al}|^{3}}(Q^{2})_{\al}\co H(\tilde{\omega}_{\al}),$$
\item $$ \da^{k}(\overline{c}_{\al} \tilde{\omega}) = \text{NICE35 } + \frac{Q^{2}\tilde{\omega}^{2}}{2|\tilde{z}_{\al}|}H(\da^{k}(\tilde{K})),$$
 \item  $$ \da^{k}(\tilde{c} \tilde{\omega}_{\al}) = \text{NICE35 }, $$
\end{itemize}
and the following collection of terms:
\subsubsection{High Order}
\underline{From $A$:}
\begin{equation*}
\begin{array}{lr}
\displaystyle 2\int Q^{2k+3}\left(\da^{k}(\tilde{K})\right)\da^{k}\co\left(H(\tilde{\omega}_{\al \al})\right) & (A^{2})
\end{array}
\end{equation*}
\underline{From $B$:}
\begin{equation*}
\begin{array}{lr}
\displaystyle -2\int Q^{2k+3}\Lambda(\da^{k+1}(\tilde{\omega}))\co\da^{k}(\tilde{K}) & (B^{2,2,1}) \\
\displaystyle -2\frac{\ep}{\tau}\|\da^{k+1} \tilde{\omega}\|^{2}_{L^{2}} & (D)
\end{array}
\end{equation*}
\underline{From $C$:}

No terms from $C$.

\subsubsection{Low Order Type I}
\underline{From $A$:}
\begin{equation*}
\begin{array}{lr}
\displaystyle 2\int 2kQ^{2k+2}Q_{\al}\left(\da^{k}(\tilde{K})\right)\da^{k-1}\co\left(H(\tilde{\omega}_{\al \al})\right) & (A^{1}) \\
\displaystyle 2\int 4Q^{2k+2}Q_{\al}\left(\da^{k}(\tilde{K})\right)\co\da^{k}\left(H(\tilde{\omega}_{\al})\right) & (A^{3}) \\
\end{array}
\end{equation*}
\underline{From $B$:}
\begin{equation*}
\begin{array}{lr}
\displaystyle - 2\int Q^{2k+2}Q_{\al}\Lambda(\da^{k}(\tilde{\omega}))\co\da^{k}(\tilde{K}) & (B^{3}) \\
\displaystyle -\int(2k+2) Q^{2k+2}Q_{\al}H(\da^{k+1}(\tilde{\omega}))\co\da^{k}(\tilde{K}) & (B^{4}) \\
\displaystyle 2\int Q^{2k+2}H(\da^{k+1}(\tilde{\omega}))Q_{\al}\co\da^{k}(\tilde{K})) & (B^{2,1})\\
\displaystyle 2k\int Q^{2k+2}H(\da^{k+1}(\tilde{\omega}))Q_{\al}\co\da^{k}(\tilde{K}) & (B^{2,3}) \\
\displaystyle -2(2k+3)\int Q^{2k+2}Q_{\al}H(\da^{k+1}(\tilde{\omega}))\co\da^{k}(\tilde{K}) & (B^{2,2,2}) \\
\end{array}
\end{equation*}
\underline{From $C$:}

No terms from $C$.

\subsubsection{Low Order Type II}

\underline{From $A$:}

No terms from $A$.

\underline{From $B$:}

\begin{equation*}
\begin{array}{lr}
\displaystyle \frac{1}{\tau}\int Q^{2k+2}\Lambda(\da^{k}(\tilde{\omega}))\frac{Q^{2}\tilde{\omega}^{2}}{|\tilde{z}_{\al}|}\co H(\da^{k}(\tilde{K}))
& (B^{1})
\end{array}
\end{equation*}

\underline{From $C$:}

\begin{equation*}
\begin{array}{lr}
\displaystyle \frac{1}{|\tilde{z}_{\al}|\tau}\int Q^{2k+4}\tilde{\omega}^{2}\da^{k}(\tilde{\omega})\co\da^{k+1}(\tilde{K}) & (C^{1})
\end{array}
\end{equation*}

We note that throughout this section we have repeatedly used the following commutator estimate for convolutions:

\begin{equation}
\label{commmol}
 \|\phi_{\delta} * (\da f g) - g \phi_{\delta} * (\da f)\|_{L^{2}} \leq C\|\da g\|_{L^{\infty}} \|f\|_{L^{2}},
\end{equation}

where the constant $C$ is independent of $\delta, f$ and $g$.

Also using this commutator estimate we can find all the cancelations we need in the previous collection of terms of low order type I and II to obtain a suitable energy estimate.

Regarding the high order terms, we will do the estimates in detail. We will see the need for the dissipative term since there are terms that escape for half of a derivative.

\begin{align*}
A^{2} + B^{2,2,1} + D & =
2\int Q^{2k+3}\left(\da^{k}(\tilde{K})\right)\co\left(H(\da^{k+2} \tilde{\omega})\right) \\
& -2\int Q^{2k+3}H(\da^{k+2}(\tilde{\omega}))\co\da^{k}(\tilde{K})
-2\ep\|\da^{k+1} \tilde{\omega}\|^{2}_{L^{2}} \\
& = 2 \int \da^{k}(\tilde{K})\left(Q^{2k+3} \co H(\da^{k+2} \tilde{\omega}) - \co \left(Q^{2k+3} H (\da^{k+2} \tilde{\omega})\right)\right)
-2\ep\|\da^{k+1} \tilde{\omega}\|^{2}_{L^{2}} \\
& \leq \|\da^{k} \tilde{K}\|_{L^{2}}\|\da Q^{2k+3}\|_{L^{\infty}}\|\da^{k+1}\tilde{\omega}\|_{L^{2}} -2\ep\|\da^{k+1} \tilde{\omega}\|^{2}_{L^{2}}
 \leq C(\ep)E^{p}(t),
\end{align*}

which is uniform in $\delta$. This proves that we can pass to the limit $\delta \to 0$.

Finally, by applying the a priori energy estimates to the new system (which only depend on $\ep$) we can pass to the limit $\ep\to 0$ since now we don't have the previous problems and $A^{2} + B^{2,2,1} = 0$.

\section{Energy with the Rayleigh-Taylor condition}
\label{SectionEnergyRT}

In this section, we prove local existence in the tilde domain, where the time of existence does not depend on the surface tension coefficient. In this theorem, we need initial data to satisfy the Rayleigh-Taylor condition as we explain in Section \ref{SectionInitialData}. This Rayleigh-Taylor condition will hold in particular if the surface tension coefficient is small enough.

\begin{theorem}
\label{localexistencetilde}
Let $k \geq 3$. Let $\tilde{z}^{0}(\al)$ be the image of a splash curve by the map $P$ parametrized in such a way that $|\da \tilde{z}^{0}(\al)| = \frac{L}{2\pi}$, where $L$ is the length of the curve in a fundamental period, and such that $\tilde{z}_{1}^{0}(\al), \tilde{z}_2^{0}(\al) \in H^{k+2}(\T)$. Let $\tilde{\varphi}(\al,0) \in H^{k+\frac{1}{2}}(\T)$ be as in \eqref{varphi} and let $\tilde{\omega}(\al,0) \in H^{k-1}(\T)$.
Then there exist a finite time $T > 0$, a time-varying curve $\tilde{z}(\al,t)  \in C([0,T];H^{k+2})$, and functions $\tilde{\omega}(\al,t) \in C([0,T];H^{k-1})$ and $\tilde{\varphi} \in C([0,T]; H^{k+\frac{1}{2}})$ providing a solution of the water wave equations (\ref{zeq} - \ref{eqomega}).
Assume that initially, the Rayleigh-Taylor condition is strictly positive.
\end{theorem}

In order to prove this theorem we will use the solutions we have obtained in  theorem \ref{localexistencetildeWRT} for $\tau>0$. We will perform energy estimates on these solutions.

\subsection{The energy}

We will define the energy for $k \geq 3$ as

\begin{align*}
 E^2_{k}(t) & = \mathcal{EE}^{2}(t) + \underbrace{\tau \frac{|\tilde{z}_{\al}|}{2}\int Q^{2k+1}\left(\da^{k}(\tilde{K})\right)^{2}}_{A}
+ \underbrace{\int Q^{2k-2}\da^{k}(\tilde{\vp})\Lambda(\da^{k}(\tilde{\vp}))}_{B} \\
& + \underbrace{|\tilde{z}_{\al}|^{2} \tau \int (\mathcal{C}\|\tilde{K}(t)\|_{H^{1}} + \tilde{K}) Q^{2k+1}\da^{k-1}(\tilde{K})\Lambda(\da^{k-1}(\tilde{K}))}_{C}
+ \underbrace{2|\tilde{z}_{\al}|  \int \mathcal{C}\|\tilde{K}(t)\|_{H^{1}} Q^{2k-2}\left(\da^{k}(\tilde{\vp})\right)^{2}}_{D} \\
& + \underbrace{|\tilde{z}_{\al}|^{2} \int \sigma Q^{2k}\left(\da^{k-1}(\tilde{K})\right)^{2}}_{E}
 + \frac{|\tilde{z}_{\al}|^{2}}{m(Q^{2k}\sigma)(t)},
\end{align*}
where $m(Q^{2k}\sigma)=\min_{\al\in\T}Q^{2k}(\tilde{z}(\al,t))\sigma(\al,t)$ and $\mathcal{C}$ is a sufficiently large constant such that $C$ is strictly positive.  Remember that $\tilde{\varphi}$ was introduced in Equation \ref{varphi}.

At this point is important to notice the following.

\begin{lemma}The following sentences hold.
\label{lemmaequivenergies}
\begin{enumerate}
\item Let $\tilde{\varphi}\in H^{3+\frac{1}{2}}$, $\tilde{\om}\in H^{2}$ and $z\in H^k$ with $k\geq 4$. Then $\tilde{\om}\in H^3$.
\item Let $\tilde{\varphi}\in H^{3+\frac{1}{2}}$, $\tilde{\om}\in H^{3}$ and $z\in H^k$ with $k\geq 5$. Then $\tilde{\om}\in H^{3.5}$.
\item Let $\tilde{\omega}\in H^{3+\frac{1}{2}}$,  and $\tilde{z}\in H^k$ with $k\geq 5$. Then $\tilde{\varphi}\in H^{3.5}$.
\end{enumerate}
\end{lemma}
This lemma shows that for  a fixed $\tau>0$ the energy of this section is equivalent to this one in section \ref{primeraenergia}. This allows us to use this energy to extend the solutions of the theorem \ref{localexistencetildeWRT} up to a time $T$ which does not depend on $\tau$ (for a small enough $\tau$).

\subsection{The energy estimates}

Again, we will only focus on the new terms ($A-E$) since the estimates for the other ones were proved in \cite{Cordoba-Cordoba-Gancedo:interface-water-waves-2d} and in \cite{Castro-Cordoba-Fefferman-Gancedo-GomezSerrano:finite-time-singularities-Euler}.

\subsubsection{$\tilde{K}$}

\begin{prop}
\begin{align*}
\tilde{K}_t & = \text{NICE3B } + \frac{Q^{2}}{2|\tilde{z}_{\al}|^{3}}H(\tilde{\omega}_{\al \al}) + \frac{1}{|\tilde{z}_{\al}|^{3}}(Q^{2})_{\al}H(\tilde{\omega}_{\al}) \\
& = \text{NICE3B } + \frac{1}{|\tilde{z}_{\al}|^{2}}H(\tilde{\vp}_{\al \al})  - \frac{1}{|\tilde{z}_{\al}|}(\tilde{K}\tilde{\vp})_{\al},
\end{align*}
where NICE3B means $$ \int Q^{j}\da^{k}(\tilde{K}) \da^{k}(NICE3B) \leq CE_{k}^{p}(t)$$
for some positive constants $C, p$ and any $j$.
\end{prop}
\begin{proof}
The first equality follows from the proof from the last section since the energies are equivalent (see Lemma \ref{lemmaequivenergies}). We now prove the second one. We begin by using the relation \eqref{varphi} to get
\begin{align*}
\tilde{K}_t & = \text{NICE3B } + \frac{Q^{2}}{|\tilde{z}_{\al}|^{2}}H\left(\left(\frac{\tilde{\vp}}{Q^{2}}\right)_{\al \al}\right)
 + \frac{Q^{2}}{|\tilde{z}_{\al}|}H\left(\left(\frac{\tilde{c}}{Q^{2}}\right)_{\al \al}\right) \\
 & + \frac{2(Q^{2})_{\al}}{|\tilde{z}_{\al}|^{2}}H\left(\left(\frac{\tilde{\vp}}{Q^{2}}\right)_{\al}\right)
 + \frac{2(Q^{2})_{\al}}{|\tilde{z}_{\al}|}H\left(\left(\frac{\tilde{c}}{Q^{2}}\right)_{\al}\right) = I + J \\
\end{align*}
We can easily see that
$$ \tilde{c}_{\al} = - \frac{\tilde{z}_{\al}}{|\tilde{z}_{\al}|^{2}} \cdot (Q^{2} BR)_{\al} = \text{NICE3B}$$
since it is at the level of $\tilde{\om}_{\al}, \tilde{z}_{\al \al}$ but we gain one derivative by multiplying by the tangential direction. This proves that
\begin{align*}
J = \text{NICE3B } + \frac{2(Q^{2})_{\al}}{|\tilde{z}_{\al}|^{2}}H\left(\frac{\tilde{\vp}_{\al}}{Q^{2}}\right).
\end{align*}
Looking now to $\tilde{c}_{\al \al}$ we can see that
\begin{align*}
\tilde{c}_{\al \al} = - \frac{\tilde{z}_{\al \al}}{|\tilde{z}_{\al}|^{2}} \cdot (Q^{2} BR)_{\al}
- \frac{\tilde{z}_{\al}}{|\tilde{z}_{\al}|^{2}} \cdot (Q^{2} BR)_{\al \al} = I_1 + I_2.
\end{align*}
Using the standard estimates, the only thing that causes trouble in $I_1$ is when all the derivatives hit $\tilde{\omega}$ and therefore
\begin{align*}
I_1 = \text{NICE3B } - K\frac{Q^{2}}{2|\tilde{z}_{\al}|}H(\tilde{\om}_{\al}).
\end{align*}
Regarding $I_2$, again, we need all the derivatives to hit $BR$ to get the most singular terms, which are

\begin{align*}
I_2 & = \text{NICE3B } - \frac{Q^{2}}{|\tilde{z}_{\al}|^{2}} \tilde{z}_{\al} \cdot \left[\frac{2}{2\pi}\int_{-\pi}^{\pi} \frac{(\tilde{z}_{\al}(\al) - \tilde{z}_{\al}(\beta))^{\perp}}{|\tilde{z}(\al)-\tilde{z}(\beta)|^{2}}\tilde{\omega}_{\al}(\al-\beta)d\beta\right] \\
& - \frac{Q^{2}}{|\tilde{z}_{\al}|^{2}} \tilde{z}_{\al} \cdot \left[\frac{1}{2\pi}\int_{-\pi}^{\pi} \frac{(\tilde{z}(\al) - \tilde{z}(\beta))^{\perp}}{|\tilde{z}(\al)-\tilde{z}(\beta)|^{2}}\tilde{\omega}_{\al \al}(\al-\beta)d\beta\right] \\
& - \frac{Q^{2}}{|\tilde{z}_{\al}|^{2}} \tilde{z}_{\al} \cdot \left[\frac{1}{2\pi}\int_{-\pi}^{\pi} \frac{(\tilde{z}_{\al \al}(\al) - \tilde{z}_{\al \al}(\beta))^{\perp}}{|\tilde{z}(\al)-\tilde{z}(\beta)|^{2}}\tilde{\omega}_{\al}(\al-\beta)d\beta\right] \\
& = \text{NICE3B } +\frac{2Q^{2}}{|\tilde{z}_{\al}|^{2}} \frac{1}{2}\frac{\tilde{z}_{\al} \cdot \tilde{z}_{\al \al}^{\perp}}{|\tilde{z}_{\al}|^{2}}H(\tilde{\omega}_{\al}) - \frac{Q^{2}}{|\tilde{z}_{\al}|^{2}}\frac{\tilde{z}_{\al} \cdot \tilde{z}_{\al \al}^{\perp}}{|\tilde{z}_{\al}|^{2}}H(\tilde{\omega}_{\al}) - \frac{Q^{2}}{|\tilde{z}_{\al}|^{2}}\frac{1}{2}\frac{\tilde{\omega}}{|\tilde{z}_{\al}|^{2}} \tilde{z}_{\al} \cdot H(\tilde{z}_{\al \al \al}^{\perp})
\end{align*}
Collecting all the terms from $I_1$ and $I_2$, we obtain
\begin{align*}
\frac{\tilde{c}_{\al \al}}{Q^{2}} & = \text{NICE3B } + \frac{1}{2|\tilde{z}_{\al}|}H((\tilde{K}\tilde{\omega})_{\al}) \\
& = \text{NICE3B } + \frac{1}{Q^{2}}H((\tilde{K} \tilde{\vp})_{\al}). \\
\end{align*}

We can finally write the total contribution as

\begin{align*}
\tilde{K}_t & = \text{NICE3B } +  \frac{Q^{2}}{|\tilde{z}_{\al}|^{2}}H\left(\frac{\tilde{\vp}_{\al \al}}{Q^{2}}\right)
 -  \frac{Q^{2}}{|\tilde{z}_{\al}|^{2}}H\left(\frac{4Q_{\al}\tilde{\vp}_{\al}}{Q^{3}}\right) \\
&  - \frac{1}{|\tilde{z}_{\al}|}(\tilde{K} \tilde{\vp})_{\al}
+ \frac{2(Q^{2})_{\al}}{|\tilde{z}_{\al}|^{2}}H\left(\frac{\tilde{\vp}_{\al}}{Q^{2}}\right) \\
& = \text{NICE3B } + \frac{Q^{2}}{|\tilde{z}_{\al}|^{2}}H\left(\frac{\tilde{\vp}_{\al \al}}{Q^{2}}\right)- \frac{1}{|\tilde{z}_{\al}|}(\tilde{K} \tilde{\vp})_{\al} \\
& = \text{NICE3B } + \frac{1}{|\tilde{z}_{\al}|^{2}}H\left(\tilde{\vp}_{\al \al}\right)- \frac{1}{|\tilde{z}_{\al}|}(\tilde{K} \tilde{\vp})_{\al}
\end{align*}

as we wanted to prove.

\end{proof}

\subsubsection{$\tilde{\varphi}$}

Throughout this section, we will use the following estimate which was proved in \cite{Castro-Cordoba-Fefferman-Gancedo-GomezSerrano:finite-time-singularities-Euler} for the case without surface tension. The proof is exactly the same for the case with it.

\begin{align*}
\varphi_{\al t} = \text{NICE2B } + \frac{\tilde{\vp} \tilde{\vp}_{\al \al}}{|\tilde{z}_{\al}|} - Q^{2} \sigma \tilde{K} + \tau \frac{Q^{2}}{2|\tilde{z}_{\al}|}\left((\tilde{K}Q)_{\al} + M_{\al}\right),
\end{align*}

where NICE2B means $$ \int Q^{j}\Lambda(\da^{k}(\tilde{\varphi})) \da^{k-1}(NICE2B) \leq CE_{k}^{p}(t)$$
for some positive constants $C, p$ and any $j$.

\subsection{Calculations of the time derivative of the energy}

Using the previous lemmas and propositions, we can get the following estimates for the derivative of the energy:
\begin{align*}
\frac{dA}{dt} & = \text{OK } + \frac{\tau}{|\tilde{z}_{\al}|}\int Q^{2k+1}\da^{k}(\tilde{K})\da^{k}(H(\tilde{\vp}_{\al \al}))
- \tau\int Q^{2k+1}\da^{k}(\tilde{K})\da^{k}((\tilde{K} \tilde{\vp})_{\al}) \\
& = \text{OK } + \frac{\tau}{|\tilde{z}_{\al}|}\int Q^{2k+1}\da^{k}(\tilde{K})\da^{k}(H(\tilde{\vp}_{\al \al}))
- \tau\int Q^{2k+1}\da^{k}(\tilde{K})\da^{k+1}( \tilde{\vp})\tilde{K}
= \text{OK } + A^1 + A^2
\end{align*}
Again, we need to be careful while computing the derivative of $B$ as in Section \ref{SectionEnergyWRT}. We obtain
\begin{align*}
\frac{dB}{dt}  & =  2\int Q^{2k-2} \Lambda(\da^{k}(\tilde{\vp})) \da^{k-1}(\tilde{\vp}_{\al t})
 + \int (Q^{2k-2})_{\al} H(\da^{k}(\tilde{\vp})) \da^{k-1}(\tilde{\vp}_{\al t}) \\
& =  \text{OK } -2\int Q^{2k-2} \Lambda(\da^{k}(\tilde{\vp})) \da^{k-1}\left(\frac{\tilde{\vp} \tilde{\vp}_{\al \al}}{|\tilde{z}_{\al}|}\right) \\
&-2\int Q^{2k-2} \Lambda(\da^{k}(\tilde{\vp})) \da^{k-1}(Q^{2} \sigma \tilde{K}) \\
&+\frac{\tau}{|\tilde{z}_{\al}|}\int Q^{2k-2} \Lambda(\da^{k}(\tilde{\vp})) \da^{k}(Q^{2} (Q\tilde{K})_{\al})\\
&-\frac{\tau}{|\tilde{z}_{\al}|}\int Q^{2k}Q_{\al} \Lambda(\da^{k}(\tilde{\vp})) \da^{k}(\tilde{K})\\
& + \int \frac{\tau}{|\tilde{z}_{\al}|}(k-1)Q^{2k+2}Q_{\al} H(\da^{k}(\tilde{\vp})) \da^{k+1}(\tilde{K}) \\
& = \text{OK } + B^{1} + B^{2} + B^{3}+B^{4} + B^{5}
\end{align*}

\subsection{Development of the derivative of the $B$ term}
%
%
%
%
%
%
%
%


We begin noticing that $B^{1} = \text{OK}$, as it was proved in \cite{Cordoba-Cordoba-Gancedo:interface-water-waves-2d}. Integrating by parts in $B^{5}$, we have that
 \begin{align*}
 B^{5} & = -\int \frac{\tau}{|\tilde{z}_{\al}|}(k-1)Q^{2k+2}Q_{\al} H(\da^{k+1}(\tilde{\vp})) \da^{k}(\tilde{K})
 \end{align*}

 Furthermore, the only singular terms arising from $B^{2}$ are when all derivatives hit either $\tilde{K}$ or $\sigma$, this gives us
\begin{align*}
B^2 = \text{OK } -2\int Q^{2k} \Lambda(\da^{k}(\tilde{\vp})) \da^{k-1}(\sigma) \tilde{K} - 2\int Q^{2k} \Lambda(\da^{k}(\tilde{\vp})) \da^{k-1}(\tilde{K}) \sigma = \text{OK } + B^{2,1} + B^{2,2}.
\end{align*}

However, the only singular term of the Rayleigh-Taylor condition that is not in $H^{k-1}$ is the one belonging to $BR_t(\tilde{z},\tilde{\omega}) \cdot \tilde{z}_{\al}$ when the time derivative hits $\omega$, this means

\begin{align*}
B^{2,1} & = \text{OK } -\tau\int Q^{2k} \Lambda(\da^{k}(\tilde{\vp})) \tilde{K} H(\da^{k}(\tilde{K}Q)) \\
& = -\tau\int Q^{2k+1} \Lambda(\da^{k}(\tilde{\vp})) \tilde{K} H(\da^{k}(\tilde{K})) \\
\end{align*}

Finally, developing $B^{3}$ we obtain
\begin{align*}
B^{3} & = \frac{\tau}{|\tilde{z}_{\al}|}\int Q^{2k-2} \Lambda(\da^{k}(\tilde{\vp})) \da^{k}(Q^{3}\tilde{K}_{\al}) \\
& + \frac{\tau}{|\tilde{z}_{\al}|}\int Q^{2k-2} \Lambda(\da^{k}(\tilde{\vp})) \da^{k}(Q^{2}Q_{\al} \tilde{K}) \\
& = B^{3,1} + B^{3,2}
\end{align*}

Modulo lower order terms we can see that
\begin{align*}
B^{3,2} = \text{OK } + \frac{\tau}{|\tilde{z}_{\al}|}\int Q^{2k}Q_{\al} \Lambda(\da^{k}(\tilde{\vp})) \da^{k}(\tilde{K})
\end{align*}

We can continue splitting $B^{3,1}$ into
\begin{align*}
B^{3,1}& = \text{OK } + \frac{\tau}{|\tilde{z}_{\al}|}\int Q^{2k+1} \Lambda(\da^{k}(\tilde{\vp})) \da^{k+1}(\tilde{K})
+ \frac{\tau}{|\tilde{z}_{\al}|}\int 3kQ^{2k}Q_{\al} \Lambda(\da^{k}(\tilde{\vp})) \da^{k}(\tilde{K}) \\
& = \text{OK }-\frac{\tau}{|\tilde{z}_{\al}|}\int Q^{2k+1} \Lambda(\da^{k+1}(\tilde{\vp})) \da^{k}(\tilde{K})
+ \frac{\tau}{|\tilde{z}_{\al}|}\int (k-1)Q^{2k}Q_{\al} \Lambda(\da^{k}(\tilde{\vp})) \da^{k}(\tilde{K}) \\
 & = \text{OK } + B^{3,1,1} + B^{3,1,2}
\end{align*}
where in the last equality we have performed an integration by parts. We can observe that

$$ B^{3,2} + B^{4} = B^{3,1,2} + B^{5} = 0, \quad  B^{3,1,1} + A^{1} = 0$$

We will now see that $B^{2,2}$ cancels with the term arising from the derivative of $E$. Taking into account the previous lemmas

$$\frac{dE}{dt} =  2\int \sigma Q^{2k}\left(\da^{k-1}(\tilde{K})\right)H(\da^{k+1}(\tilde{\vp})) = \text{OK } -B^{2,2}$$

Finally, we will see that the contributions from the time derivatives of $C$ and $D$ cancel $B^{2,1}$ and $A^{2}$. We start by noticing that, modulo lower order terms $A^{2} = B^{2,1}$. Furthermore

\begin{align*}
\frac{dC}{dt} & =  \text{OK } + 2\tau \int (\mathcal{C}\|\tilde{K}(t)\|_{H^{1}} + \tilde{K}) Q^{2k+1}H(\da^{k+1}(\tilde{\vp}))\Lambda(\da^{k-1}(\tilde{K})) \\
\frac{dD}{dt} & =  \text{OK } + 2\tau \int \mathcal{C}\|\tilde{K}(t)\|_{H^{1}} Q^{2k+1}\left(\da^{k}(\tilde{\vp})\right) \da^{k+1}(\tilde{K}),
\end{align*}

which, by integration by parts results in

\begin{align*}
\frac{dC}{dt} + \frac{dD}{dt} + A^{2} + B^{2,1} = \text{OK}.
\end{align*}

Adding all the contributions, we can bound the derivative in time of the energy by a power of the energy.

\appendix
\section{Helpful estimates for the Birkhoff-Rott operator}
In this Appendix we will prove some of the estimates used throughout the paper for the sake of clarity to the reader.

We begin with a classical decomposition of the Birkhoff-Rott operator. We should notice that we can write it in the following ways. On one hand:

\begin{align*}
BR(\tilde{z},\tilde{\omega})(\al)&=\frac{1}{2\pi}\int_{-\pi}^\pi \left(\frac{(\tilde{z}(\al)-\tilde{z}(\al-\beta))^\bot}{|\tilde{z}(\al)-\tilde{z}(\al-\beta)|^2}-\frac{\dpa \tilde{z}(\al)}{2|\da \tilde{z}(\al)|^2\tan(\beta/2)}\right)\tilde{\omega}(\al-\beta)d\beta \\
& +\frac{1}{2\pi}\int_{-\pi}^\pi \left(\frac{\dpa \tilde{z}(\al)}{2|\da \tilde{z}(\al)|^2\tan(\beta/2)}\right)\tilde{\omega}(\al-\beta)d\beta \\
& = \frac{1}{2\pi}\int_{-\pi}^\pi \left(\frac{(\tilde{z}(\al)-\tilde{z}(\al-\beta))^\bot}{|\tilde{z}(\al)-\tilde{z}(\al-\beta)|^2}-\frac{\dpa \tilde{z}(\al)}{2|\da \tilde{z}(\al)|^2\tan(\beta/2)}\right)\tilde{\omega}(\al-\beta)d\beta+\frac{\dpa \tilde{z}(\al)}{2|\da \tilde{z}(\al)|^2}H(\tilde{\omega})(\al) \\
& = \frac{\dpa \tilde{z}(\al)}{2|\da \tilde{z}(\al)|^2}H(\tilde{\omega})(\al) + \text{l.o.t}(\tilde{\omega}).
\end{align*}

On the other hand:

\begin{align*}
BR(\tilde{z},\tilde{\omega})(\al)&=\frac{1}{2\pi}\int_{-\pi}^\pi \frac{(\tilde{z}(\al)-\tilde{z}(\al-\beta))^\bot}{|\tilde{z}(\al)-\tilde{z}(\al-\beta)|^2}
(\tilde{\omega}(\al-\beta)-\tilde{\omega}(\al))d\beta
+ \frac{\tilde{\omega}(\al)}{2\pi}\int_{-\pi}^\pi \frac{(\tilde{z}(\al)-\tilde{z}(\al-\beta))^\bot}{|\tilde{z}(\al)-\tilde{z}(\al-\beta)|^2}d\beta\\
&=\frac{1}{2\pi}\int_{-\pi}^\pi \frac{(\tilde{z}(\al)-\tilde{z}(\al-\beta))^\bot}{|\tilde{z}(\al)-\tilde{z}(\al-\beta)|^2}
(\tilde{\omega}(\al-\beta)-\tilde{\omega}(\al))d\beta \\
& + \frac{\tilde{\omega}(\al)}{2\pi}\int_{-\pi}^\pi (\tilde{z}(\al)-\tilde{z}(\al-\beta))^\bot\left(\frac{1}{|\tilde{z}(\al)-\tilde{z}(\al-\beta)|^2} - \frac{1}{4|\tilde{z}_{\al}(\al)|^{2}\sin^{2}\left(\frac{\beta}{2}\right)}\right)d\beta \\
&+ \frac{\tilde{\omega}(\al)}{2|\tilde{z}_{\al}(\al)|^{2}} \Lambda(\tilde{z}^{\perp}(\al))\\
& = \frac{\tilde{\omega}(\al)}{2|\tilde{z}_{\al}(\al)|^{2}} \Lambda(\tilde{z}^{\perp}(\al))+ \text{l.o.t}(\tilde{z}).
\end{align*}

See \cite{Ambrose-Masmoudi:zero-surface-tension-2d-waterwaves}, \cite{Cordoba-Cordoba-Gancedo:interface-water-waves-2d} for more details concerning the lower order terms.

We will now prove energy estimates for the Birkhoff-Rott integral, showing that it is as regular as $\da \tilde{z}$. The proof is taken from \cite[Section 6]{Cordoba-Cordoba-Gancedo:interface-water-waves-2d}.

\begin{lemma}The following estimate holds
\begin{eqnarray}\label{nsibr}
\|BR(\tilde{z},\tilde{\omega})\|_{H^k}\leq
C(\|\F(\tilde{z})\|^2_{L^\infty}+\|\tilde{z}\|^2_{H^{k+1}}+\|\tilde{\omega}\|^2_{H^{k}})^j,
\end{eqnarray}
for $k\geq 2$, where $C$ and $j$ are constants independent of $\tilde{z}$
and $\tilde{\omega}$.
\end{lemma}
\begin{rem}
Using this estimate for $k=2$ we find easily that
\begin{eqnarray}\label{nliibr}
\|\da BR(\tilde{z},\tilde{\omega})\|_{L^\infty}\leq
C(\|\F(\tilde{z})\|^2_{L^\infty}+\|\tilde{z}\|^2_{H^{3}}+\|\tilde{\omega}\|^2_{H^{2}})^j.
\end{eqnarray}
\end{rem}
\begin{proof}
We shall present the proof for $k=2$. Let us write
\begin{align*}
\begin{split}
BR(\tilde{z},\tilde{\omega})(\al,t)&=\frac{1}{2\pi}\int_{-\pi}^\pi C_1(\al,\beta)\tilde{\omega}(\al-\beta)d\beta
+\frac{\dpa \tilde{z}(\al)}{2|\da \tilde{z}(\al)|^2}H(\tilde{\omega})(\al)
\end{split}
\end{align*}
where $C_1$ is given by

\begin{equation}\label{fb}
C_1(\al,\beta)=\frac{(\tilde{z}(\al)-\tilde{z}(\al-\beta))^\bot}{|\tilde{z}(\al)-\tilde{z}(\al-\beta)|^2}-\frac{\dpa \tilde{z}(\al)}{2|\da \tilde{z}(\al)|^2\tan(\beta/2)},
\end{equation}

We shall show that $\|C_1\|_{L^\infty}\leq C\|\F(\tilde{z})\|^2_{L^{\infty}}\|\tilde{z}\|^2_{C^2}$. To do so we split $C_1=D_1+D_2+D_3$ where
$$
D_1=\frac{(\tilde{z}(\al)-\tilde{z}(\al-\beta)-\da \tilde{z}(\al)\beta)^\bot}{|\tilde{z}(\al)-\tilde{z}(\al-\beta)|^2},\quad D_2=\dpa \tilde{z}(\al)[\frac{\beta}{|\tilde{z}(\al)-\tilde{z}(\al-\beta)|^2}-\frac{1}{|\da \tilde{z}(\al)|^2\beta}],
$$
and
$$
D_3=\frac{\dpa \tilde{z}(\al)}{|\da \tilde{z}(\al)|^2}[\frac1\beta-\frac{1}{2\tan(\beta/2)}].
$$
The inequality \begin{equation}\label{ncll}|\tilde{z}(\al)-\tilde{z}(\al-\beta)-\da
\tilde{z}(\al)\beta|\leq \|\tilde{z}\|_{C^2}|\beta|^2\end{equation} yields easily
$|D_1|\leq \|\tilde{z}\|_{C^2}\|\F(\tilde{z})\|^2_{L^{\infty}}$.

Then we can rewrite $D_2$ as follows:
$$
D_2=\dpa \tilde{z}(\al)[\frac{(\da \tilde{z}(\al)\beta-(\tilde{z}(\al)-\tilde{z}(\al-\beta)))\cdot (\da \tilde{z}(\al)\beta+(\tilde{z}(\al)-\tilde{z}(\al-\beta)))}{|\tilde{z}(\al)-\tilde{z}(\al-\beta)|^2|\da \tilde{z}(\al)|^2\beta}],
$$
and, in particular, we have
$$
|D_2|\leq \frac{|\da \tilde{z}(\al)\beta-(\tilde{z}(\al)-\tilde{z}(\al-\beta))|(|\da \tilde{z}(\al)\beta|+|\tilde{z}(\al)-\tilde{z}(\al-\beta)|)}{|\tilde{z}(\al)-\tilde{z}(\al-\beta)|^2|\da \tilde{z}(\al)||\beta|}.
$$
Using \eqref{ncll} we find that $|D_2|\leq 2
\|\tilde{z}\|_{C^2}\|\F(\tilde{z})\|^2_{L^{\infty}}$.

Next let us observe that since $\beta \in [-\pi,\pi]$ gives $|D_3|\leq
C\|\F(\tilde{z})\|_{L^{\infty}}$.

The boundedness of the term $C_1$ in
$L^\infty$ gives us easily
\begin{equation}\label{brl2}
\|BR(\tilde{z},\tilde{\omega})\|_{L^2}\leq C\|\F(\tilde{z})\|^2_{L^\infty}\|\tilde{z}\|^2_{C^2}\|\tilde{\omega}\|_{L^2}.
\end{equation}
 In $\da^2BR(\tilde{z},\tilde{\omega})$, the most singular terms are given by
$$
P_1(\al)=\frac{1}{2\pi}PV\int_{-\pi}^\pi\da^2\tilde{\omega}(\al-\beta)\frac{(\tilde{z}(\al)-\tilde{z}(\al-\beta))^{\bot}}{|\tilde{z}(\al)-\tilde{z}(\al-\beta)|^2}d\beta,
$$
$$
P_2(\al)=\frac{1}{2\pi}PV\int_{-\pi}^\pi\tilde{\omega}(\al-\beta)\frac{(\da^2\tilde{z}(\al)-\da^2\tilde{z}(\al-\beta))^{\bot}}{|\tilde{z}(\al)-\tilde{z}(\al-\beta)|^2}d\beta,
$$
$$
P_3(\al)=-\frac{1}{\pi}PV\int_{-\pi}^\pi\tilde{\omega}(\al-\beta)\frac{(\tilde{z}(\al)-\tilde{z}(\al-\beta))^\bot}{|\tilde{z}(\al)-\tilde{z}(\al-\beta)|^4}\big(\tilde{z}(\al)-\tilde{z}(\al-\beta))\cdot(\da^2 \tilde{z}(\al)-\da^2 \tilde{z}(\al-\beta))\big)d\beta.
$$
Again we have the expression
\begin{align*}
\begin{split}
P_1(\al)&=\frac{1}{2\pi}\int_{-\pi}^\pi C_1(\al,\beta)\da^2\tilde{\omega}(\al-\beta)d\beta
+\frac{\dpa
\tilde{z}(\al)}{2|\da \tilde{z}(\al)|^2}H(\da^2\tilde{\omega})(\al)d\al,
\end{split}
\end{align*} giving us
\begin{align}\label{q1nl22dbr}
\begin{split}
|P_1(\al)|&\leq C\|\F(\tilde{z})\|^j_{L^\infty}\|\tilde{z}\|^j_{C^2}(\|\da^2\tilde{\omega}\|_{L^2}+|H(\da^2\tilde{\omega})(\alpha)|).
\end{split}
\end{align}
Next let us write $P_2=Q_1+Q_2+Q_3$ where
$$
Q_1(\al)=\frac{1}{2\pi}\int_{-\pi}^\pi(\tilde{\omega}(\al-\beta)-\tilde{\omega}(\al))\frac{(\da^2 \tilde{z}(\al)-\da^2 \tilde{z}(\al-\beta))^{\bot}}{|\tilde{z}(\al)-\tilde{z}(\al-\beta)|^2}d\beta,
$$
$$
Q_2(\al)=\frac{\tilde{\omega}(\al)}{2\pi}\int_{-\pi}^\pi(\da^2\tilde{z}(\al)-\da^2\tilde{z}(\al-\beta))^{\bot}\big(\frac{1}{|\tilde{z}(\al)-\tilde{z}(\al-\beta)|^2}-\frac{1}{|\da \tilde{z}(\al)|^2|\beta|^2}\big)d\beta,
$$
$$
Q_3(\al)=\frac{1}{2\pi}\frac{\tilde{\omega}(\al)}{|\da \tilde{z}(\al)|^2}\int_{-\pi}^\pi\!\!(\da^2\tilde{z}(\al)\!-\!\da^2\tilde{z}(\al\!-\!\beta))^{\bot}\big(\frac{1}{|\beta|^2}\!-\!\frac{1}{4\sin^2(\beta/2)}\big)d\beta\!+\!\frac{1}{2}\frac{\tilde{\omega}(\al)}{|\da \tilde{z}(\al)|^2}\la (\da^2\tilde{z^{\bot}})(\al),
$$
where $\la=\da H$.

Using that
$$|\da^2\tilde{z}(\al)-\da^2\tilde{z}(\al-\beta)|\leq|\beta|^\delta\|\tilde{z}\|_{C^{2,\delta}},$$
we get $|Q_1(\al)|+|Q_2(\al)|\leq
\|\tilde{\omega}\|_{C^1}\|\F(\tilde{z})\|^{j}\|\tilde{z}\|^j_{C^{2,\delta}},$ while for
$Q_3$ we have
\begin{align*}
|Q_3(\al)|&\leq C\|\tilde{\omega}\|_{L^\infty}\|\F(\tilde{z})\|_{L^\infty}(\|\tilde{z}\|_{C^2}+|\la (\da^2\tilde{z}^{\bot})(\al)|),
\end{align*}
that is
\begin{equation}\label{q2nl22dbr}
|P_2(\al)|\leq (1+|\la (\da^2\tilde{z}^{\bot})(\al)|)\|\tilde{\omega}\|_{C^1}\|\F(\tilde{z})\|^{j}\|\tilde{z}\|^j_{C^{2,\delta}}.
\end{equation}
Let us now consider  $P_3=Q_4+Q_5+Q_6+Q_7+Q_8+Q_9$, where
\begin{align*}
Q_4=\frac{-1}{\pi}\int_{-\pi}^\pi(\tilde{\omega}(\al\!-\!\beta)\!-\!\tilde{\omega}(\al))\frac{(\tilde{z}(\al)\!-\!\tilde{z}(\al\!-\!\beta))^\bot}{|\tilde{z}(\al)\!-\!\tilde{z}(\al\!-\!\beta)|^4}\big((\tilde{z}(\al)\!-\!\tilde{z}(\al\!-\!\beta))\!\cdot\!(\da^2 \tilde{z}(\al)\!-\!\da^2 \tilde{z}(\al\!-\!\beta))\big)d\beta,
\end{align*}
$$Q_5=-\frac{\tilde{\omega}(\al)}{\pi}\int_{-\pi}^\pi\frac{(\tilde{z}(\al)\!-\!\tilde{z}(\al\!-\!\beta)\!-\!\da \tilde{z}(\al)\beta)^{\bot}}{|\tilde{z}(\al)\!-\!\tilde{z}(\al\!-\!\beta)|^4}\big((\tilde{z}(\al)\!-\!\tilde{z}(\al\!-\!\beta))\cdot(\da^2 \tilde{z}(\al)\!-\!\da^2 \tilde{z}(\al\!-\!\beta))\big)d\beta,
$$
$$
Q_6=-\frac{\tilde{\omega}(\al)\dpa \tilde{z}(\al)}{\pi}\int_{-\pi}^\pi\frac{\beta(\tilde{z}(\al)\!-\!\tilde{z}(\al\!-\!\beta)\!-\!\da \tilde{z}(\al)\beta)\cdot(\da^2 \tilde{z}(\al)\!-\!\da^2 \tilde{z}(\al\!-\!\beta))}{|\tilde{z}(\al)\!-\!\tilde{z}(\al\!-\!\beta)|^4}d\beta,
$$
$$
Q_7=-\frac{\tilde{\omega}(\al)\dpa \tilde{z}(\al)}{\pi}\da \tilde{z}(\al)\cdot\!\!\int_{-\pi}^\pi\beta^2(\da^2 \tilde{z}(\al)-\da^2 \tilde{z}(\al-\beta))\big(\frac{1}{|\tilde{z}(\al)\!-\!\tilde{z}(\al\!-\!\beta)|^4}-\frac{1}{|\da \tilde{z}(\al)|^4|\beta|^4}\big)d\beta,
$$
$$
Q_8=-\frac{\tilde{\omega}(\al)\dpa \tilde{z}(\al)}{\pi|\da \tilde{z}(\al)|^4}\da \tilde{z}(\al)\cdot\!\!\int_{-\pi}^\pi(\da^2 \tilde{z}(\al)-\da^2 \tilde{z}(\al-\beta))\big(\frac{1}{|\beta|^2}-\frac{1}{4\sin^2(\beta/2)}\big)d\beta,
$$
and
$$
Q_9=-\frac{\tilde{\omega}(\al)\dpa \tilde{z}(\al)}{|\da \tilde{z}(\al)|^4}\da \tilde{z}(\al)\cdot \la(\da^2 \tilde{z}(\al)).
$$
Proceeding as before we get
\begin{equation*}
|P_3(\al)|\leq C(1+|\la (\da^2\tilde{z})(\al)|)\|\tilde{\omega}\|_{C^1}\|\F(\tilde{z})\|^j_{L^\infty}\|\tilde{z}\|^j_{C^{2,\delta}},
\end{equation*}
which together with \eqref{q1nl22dbr} and \eqref{q2nl22dbr} gives
us the estimate
$$
|(P_1+P_2+P_3)(\al)|\leq C(1+|\la (\da^2\tilde{z})(\al)|+|H(\da^2\tilde{\omega})(\alpha)|)\|\tilde{\omega}\|_{C^1}(\|\F(\tilde{z})\|^j_{L^\infty}+\|\tilde{z}\|^j_{H^3}).
$$
For the rest of the terms in $\da^2 BR(\tilde{z},\tilde{\omega})$ we obtain
analogous estimates allowing us to conclude the equality
\begin{equation*}\label{enl22dbr}
\|\da^2 BR(\tilde{z},\tilde{\omega})\|_{L^2}\leq C(1+\|\da^3\tilde{z}\|_{L^2}+\|\da^2\tilde{\omega}\|_{L^2}) \|\tilde{\omega}\|_{C^1}\|\F(\tilde{z})\|^j_{L^\infty}\|\tilde{z}\|^j_{C^{2,\delta}}.
\end{equation*}
Finally the Sobolev inequalities yield \eqref{nsibr} for $k=2$.
\end{proof}

\begin{lemma}The following estimate will also be helpful
\begin{eqnarray}\label{nsibr}
\|\da BR(\tilde{z},\tilde{\omega}) \cdot \da \tilde{z}\|_{H^k}\leq
C(\|\F(\tilde{z})\|^2_{L^\infty}+\|\tilde{z}\|^2_{H^{k+2}}+\|\tilde{\omega}\|^2_{H^{k}})^j,
\end{eqnarray}
for $k\geq 2$, where $C$ and $j$ are constants independent of $\tilde{z}$
and $\tilde{\omega}$.
\end{lemma}
\begin{proof}
In $\da BR(\tilde{z},\tilde{\omega}) \cdot \da \tilde{z}$, the most singular terms are given by
$$
R_1(\al)=\frac{1}{2\pi}PV\int_{-\pi}^\pi\da\tilde{\omega}(\al-\beta)\frac{\da \tilde{z}(\al) \cdot (\tilde{z}(\al)-\tilde{z}(\al-\beta))^{\bot}}{|\tilde{z}(\al)-\tilde{z}(\al-\beta)|^2}d\beta,
$$
$$
R_2(\al)=\frac{1}{2\pi}PV\int_{-\pi}^\pi\tilde{\omega}(\al-\beta)\frac{\da \tilde{z}(\al) \cdot (\da\tilde{z}(\al)-\da\tilde{z}(\al-\beta))^{\bot}}{|\tilde{z}(\al)-\tilde{z}(\al-\beta)|^2}d\beta,
$$
$$
R_3(\al)=-\frac{1}{\pi}PV\int_{-\pi}^\pi\tilde{\omega}(\al-\beta)\frac{\da \tilde{z}(\al) \cdot (\tilde{z}(\al)-\tilde{z}(\al-\beta))^\bot}{|\tilde{z}(\al)-\tilde{z}(\al-\beta)|^4}\big(\tilde{z}(\al)-\tilde{z}(\al-\beta))\cdot(\da \tilde{z}(\al)-\da \tilde{z}(\al-\beta))\big)d\beta.
$$

$R_2$ can be estimated in the same way as $P_2$. Regarding $R_1$, one can write it as
$$
R_1(\al)=\frac{1}{2\pi}PV\int_{-\pi}^\pi\da\tilde{\omega}(\al-\beta)\left[\frac{\da \tilde{z}(\al) \cdot (\tilde{z}(\al)-\tilde{z}(\al-\beta))^{\bot}}{|\tilde{z}(\al)-\tilde{z}(\al-\beta)|^2} - \frac{\da \tilde{z}_{\al} \cdot \da \tilde{z}_{\al}^{\bot}}{|\tilde{z}_{\al}(\al)|^{2}}\right]d\beta.
$$

Now, since $\da \tilde{\omega}(\al-\beta) = -\partial_{\beta} \tilde{\omega}(\al-\beta)$, one can integrate by parts and bound the resulting kernel (which has order -1) giving
$$ |R_1| \leq C(\|\F(\tilde{z})\|^2_{L^\infty}+\|\tilde{z}\|^2_{H^{k+2}}+\|\tilde{\omega}\|^2_{H^{k}})^j.$$

Finally, $R_3$ can be written in the form

\begin{align*}
R_3(\al)=-\frac{1}{\pi}PV\int_{-\pi}^\pi\tilde{\omega}(\al-\beta)&\left[\frac{\da \tilde{z}(\al) \cdot (\tilde{z}(\al)-\tilde{z}(\al-\beta))^\bot}{|\tilde{z}(\al)-\tilde{z}(\al-\beta)|^4}\big(\tilde{z}(\al)-\tilde{z}(\al-\beta))\cdot(\da \tilde{z}(\al)-\da \tilde{z}(\al-\beta))\big)\right. \\
&\left.- \frac{\da \tilde{z}(\al) \cdot \da \tilde{z}(\al)^{\bot}}{\beta|\tilde{z}_{\al}(\al)|^{4}}\da \tilde{z}(\al) \cdot \partial_{\al}^{2} \tilde{z}(\al)\right]d\beta,
\end{align*}

and bound $R_3$ by the kernel (which has order 0) in $L^{\infty}$ norm and $\omega$ in $L^{2}$ norm. This completes the proof.

\end{proof}

Then, the following corollary is immediate
\begin{corollary}
\begin{eqnarray}\label{nsibr}
\|\tilde{c}_{\al}\|_{H^k}\leq
C(\|\F(\tilde{z})\|^2_{L^\infty}+\|\tilde{z}\|^2_{H^{k+2}}+\|\tilde{\omega}\|^2_{H^{k}})^j,
\end{eqnarray}
for $k\geq 2$, where $C$ and $j$ are constants independent of $\tilde{z}$
and $\tilde{\omega}$.
\end{corollary}

\subsection*{{\bf Acknowledgements}}

\smallskip

 AC, DC, FG and JGS were partially supported by the grant {\sc MTM2011-26696} of the MCINN (Spain) and
the grant StG-203138CDSIF  of the ERC. FG acknowledges support from the Ram\'on y Cajal program. CF was partially supported by
NSF grant DMS-0901040. We are grateful for the support of the Fundaci\'on General del CSIC.


\begin{tabular}{ll}
\textbf{Angel Castro} &  \\
{\small D\'epartement de Math\'ematiques et Applications} & \\
{\small \'Ecole Normale Sup\'erieure} &\\
{\small 45, Rue d'Ulm, 75005 Paris} & \\
{\small Email: castro@dma.ens.fr} & \\
   & \\
\textbf{Diego C\'ordoba} &  \textbf{Charles Fefferman}\\
{\small Instituto de Ciencias Matem\'aticas} & {\small Department of Mathematics}\\
{\small Consejo Superior de Investigaciones Cient\'ificas} & {\small Princeton University}\\
{\small C/ Nicol\'{a}s Cabrera, 13-15} & {\small 1102 Fine Hall, Washington Rd, }\\
{\small Campus Cantoblanco UAM, 28049 Madrid} & {\small Princeton, NJ 08544, USA}\\
{\small Email: dcg@icmat.es} & {\small Email: cf@math.princeton.edu}\\
 & \\
\textbf{Francisco Gancedo} &  \textbf{Javier G\'omez-Serrano}\\
{\small Departamento de An\'alisis Matem\'atico} & {\small Instituto de Ciencias Matem\'aticas}\\
{\small Universidad de Sevilla} & {\small Consejo Superior de Investigaciones Cient\'ificas}\\
{\small C/ Tarfia, s/n } & {\small C/ Nicol\'{a}s Cabrera, 13-15} \\
{\small Campus Reina Mercedes, 41012 Sevilla}  & {\small Campus Cantoblanco UAM, 28049 Madrid} \\
{\small Email: fgancedo@us.es} & {\small Email: javier.gomez@icmat.es}\\
\end{tabular}

\end{document}